\documentclass[12pt,a4paper,reqno]{amsart}
\usepackage[dvips]{graphicx}
\usepackage{verbatim}
\usepackage{hyperref}
\usepackage{color}
\usepackage{amssymb}
\usepackage{latexsym}
\usepackage{amsmath}
\usepackage{bbm}
\usepackage{amsthm}
\usepackage{amsrefs}

\newcommand{\R}{\mathbbm{R}}    
\newcommand{\Z}{\mathbbm{Z}}    
\providecommand{\ab}[1]{\vert #1\vert}		
\newcommand{\N}{\mathbbm{N}}    
\newcommand{\te}{\theta}
\newcommand{\vphi}{\varphi}
\newcommand{\la}{\lambda}
\newcommand{\eps}{\varepsilon}
\providecommand{\norma}[1]{\Vert #1 \Vert}          
\providecommand{\Norma}[1]{\Bigl\Vert #1 \Bigr\Vert}          
\newcommand{\cp}{\mathcal{C}}  
\providecommand{\suma}[2]{\sum\limits_{#1}^{#2}}  
\providecommand{\abs}[1]{\left\vert #1 \right\vert}           
\providecommand{\ab}[1]{\vert #1\vert}           
\renewcommand{\leq}{\leqslant}
\renewcommand{\geq}{\geqslant}
\renewcommand\Re{\operatorname{\mathfrak{Re}}}

\numberwithin{equation}{section}
\theoremstyle{plain}
\newtheorem{teo}{Theorem}[section]
\newtheorem{cor}[teo]{Corollary}
\newtheorem{lemma}[teo]{Lemma}
\newtheorem{prop}[teo]{Proposition}

\theoremstyle{definition}
\newtheorem{define}[teo]{Definition}

\newtheorem{remark}[teo]{Remark}

\begin{document}
\title[On extremizing sequences for the cone]{On extremizing sequences for the adjoint restriction inequality on the cone}
\author{Ren\'e Quilodr\'an}
\address{Department of Mathematics, University of California, Berkeley, CA 94720-3840, USA}
\email{rquilodr@math.berkeley.edu}
\thanks{Research supported in part by NSF grant DMS-0901569.}

\begin{abstract}
 It is known that extremizers for the $L^2$ to $L^6$ adjoint Fourier restriction inequality on the cone in $\R^3$ exist. Here, we
show that nonnegative extremizing sequences are precompact, after the application of symmetries of the cone. If we use the
knowledge of the exact form of the extremizers, as found by Carneiro, then we can show that nonnegative extremizing sequences
converge, after the application of symmetries.
\end{abstract}

\maketitle
\section{Introduction}
We study the properties of extremizing sequences for the Fourier restriction inequality on the cone in dimension $3$ for which 
the adjoint restriction inequality can be rewritten equivalently as a convolution inequality. Carneiro \cite{Ca}, using the 
method developed by Foschi \cite{Fo}, found the exact form of the extremizers for the adjoint Fourier restriction inequalities in
dimensions $3$ and $4$, but there seems to be no mention in the literature as to whether extremizing sequences are precompact
after
appropriate rescaling\footnote{Fanelli, Vega and Visciglia \cite{FVV2} answer this question in the nonendpoint case and appeared
while this manuscript was
being prepared. We comment on that later in the Introduction.}. That is the question we try to answer in this paper using the
methods developed by Christ and Shao \cite{CS} to analyze the corresponding inequality for the sphere in three dimensions.

We denote $\Gamma^2=\{(y,y')\in\R^2\times \R:y'=\ab{y}\}$, the cone in $\R^3$. A function $f$ on $\Gamma^2$ can be identified, 
and we will do so, with a function from $\R^2$ to $\R$. On $\Gamma^2$ we consider the measure
$\sigma(y,y')=\delta(y'-\ab{y})\ab{y}^{-1}dydy'$, that is, for a function $f$ on the cone
\[\int_{\Gamma^2} fd\sigma=\int_{\R^2} f(y)\frac{dy}{\ab{y}}.\]
We will denote the $L^p(\Gamma^2,\sigma)$ norm of a function $f$ as $\norma{f}_{L^p(\Gamma^2)},\,\norma{f}_{L^p(\sigma)}$ or 
$\norma{f}_p$.

The extension or adjoint Fourier restriction operator for the cone is given by
\begin{equation}
 \label{fourier-extension-operator}
 Tf(x,t)=\int_{\R^2} e^{ix\cdot y} e^{it\ab{y}} f(y)\ab{y}^{-1}dy,
\end{equation}
where $(x,t)\in\R^2\times \R$ and $f\in \mathcal S(\R^2)$. With the Fourier transform $\hat g(\xi)=\break\int_{\R^3} e^{-ix\cdot
\xi}g(x)dx$, we see that $Tf(x,t)=\widehat{f\sigma}(-x,-t)$.

A well known bound \cite{Str} for $Tf$ is given in the following theorem:
\begin{teo}
\label{adjoint-fourier-restriction}
 There exists $C<\infty$ such that for all $f\in L^2(\Gamma^2)$ the following inequality holds:
\begin{equation}
\label{restriction-cone-with-c}
 \norma{Tf}_{L^6(\R^3)}\leq C\norma{f}_{L^2(\Gamma^2)}.
\end{equation}
\end{teo}

Denote by $\mathbf{C}$ the best constant in \eqref{restriction-cone-with-c}, that is,
\begin{equation}
 \label{best-constant}
\mathbf{C}=\sup\limits_{0\neq f\in L^2(\Gamma^2)}\frac{\norma{Tf}_{L^6(\R^3)}}{\norma{f}_{L^2(\Gamma^2)}}.
\end{equation}

The use of the Fourier transform allows us to write \eqref{restriction-cone-with-c} in ``convolution form'', namely
\begin{align}
 \norma{Tf}_{L^{6}(\R^{3})}^{3}&=\norma{(Tf)^{3}}_{L^2(\R^{3})}=\norma{(\widehat{f\sigma})^{3}}_{L^2(\R^{3})}
 =\norma{(f\sigma*f\sigma*f\sigma)\hat{}\,}_{L^2(\R^{3})}\nonumber\\
 \label{convolution-form}
 &=(2\pi)^{3/2}\norma{f\sigma*f\sigma*f\sigma}_{L^2(\R^{3})},
\end{align}
thus $\norma{T(f)}_{L^6}\leq \norma{T(\ab{f})}_{L^6}$. This implies that if $\{f_n\}_{n\in\N}$ is an extremizing sequence then so
is $\{\ab{f_n}\}_{n\in\N}$.

In what follows we will restrict attention to nonnegative functions $f\in L^2(\Gamma^2)$.

\begin{define}
\label{def-ext-sequence}
 An extremizing sequence for inequality \eqref{restriction-cone-with-c} is a sequence $\{f_n\}_{n\in\N}$ of functions in
$L^2(\Gamma^2)$ satisfying $\norma{f_n}_{L^2(\Gamma^2)}\leq 1$, such that $\norma{Tf_n}_{L^6(\R^{3})}\to \mathbf{C}$ as
$n\to\infty$.
 
 An extremizer for \eqref{restriction-cone-with-c} is a function $f\neq 0$ which satisfies 
 $\norma{Tf}_{L^6(\R^{3})}=\mathbf{C}\norma{f}_{L^2}$.
\end{define}

The main theorem of this paper is as follows:

\begin{teo}
\label{main-precompactness}
 Any extremizing sequence of nonnegative functions in $L^2(\Gamma^2)$ for inequality \eqref{restriction-cone-with-c} is 
 precompact up to symmetries, that is, every subsequence of an extremizing sequence has a sub-subsequence that converges in
$L^2(\Gamma^2)$ after the application of symmetries of the cone.
\end{teo}

The symmetries of the cone we refer to are dilations and Lorentz transformations that will be studied in Section 
\ref{lorentz-invariance}, and Theorem \ref{main-precompactness} will be stated in a more precise form as Theorem
\ref{precompactness}.

With the knowledge of the exact form of the extremizers to \eqref{restriction-cone-with-c} given by Carneiro in \cite{Ca}, 
one can improve Theorem \ref{main-precompactness} to obtain the following:

\begin{teo}
 \label{better-precompactness}
 Any extremizing sequence of nonnegative functions in $L^2(\Gamma^2)$ for inequality \eqref{restriction-cone-with-c} 
 converges in $L^2(\Gamma^2)$, after the application of symmetries of the cone.
\end{teo}

Define the function $g$ by its Fourier transform as $\hat g(y)=f(y)\ab{y}^{-1}$. Then
\begin{equation}
 \label{equivalent-form}
e^{it\sqrt{-\Delta}}g(x):=\frac{1}{(2\pi)^2}\int_{\R^2} e^{ix\cdot y}e^{it\ab{y}}\hat g(y)dy=\frac{1}{(2\pi)^2}Tf(x,t),
\end{equation}
and
\[\norma{g}_{\dot H^{\frac{1}{2}}(\R^2)}=\norma{f}_{L^2(\Gamma^2)},\]
where we used the $\dot H^{\frac{1}{2}}(\R^2)$ norm
\[\norma{g}_{\dot H^{\frac{1}{2}}(\R^2)}^2=\int_{\R^2}\ab{\hat g(y)}^2\ab{y}dy.\]
We see that 
\begin{equation}
 \label{equivalence-operator}
 (2\pi)^{-2}\norma{Tf}_{L^6(\R^2)}\norma{f}_{L^2(\Gamma^2)}^{-1}=\norma{e^{it\sqrt{-\Delta}}g}_{L^6(\R^3)}\norma{g}_{\dot
H^{\frac{1}{2}}(\R^2)}^{-1}
\end{equation}
and \eqref{restriction-cone-with-c} is equivalent to
\begin{equation}
 \label{Strichatz-estimate}
\norma{e^{it\sqrt{-\Delta}}g}_{L^6_{x,t}{(\R^3)}}\leq \frac{C}{(2\pi)^2}\norma{g}_{\dot H^{\frac{1}{2}}(\R^2)}.
\end{equation}

From \eqref{equivalence-operator}, $\{f_n\}_{n\in\N}$ is an extremizing sequence for \eqref{restriction-cone-with-c} if and only
if $\{g_n\}_{n\in\N}$, with $\hat g_n(y)=f_n(y)\ab{y}^{-1}$, is an extremizing sequence for \eqref{Strichatz-estimate}.

The problem of computing the best constant in \eqref{restriction-cone-with-c} and the exact form of the extremizers was solved 
by Carneiro in \cite{Ca}. With the normalization of the Fourier transform discussed earlier, Carneiro proves the following:
\begin{teo}\cite{Ca}.
\label{extremizers-cone}
For all $f\in L^2(\Gamma^2)$,
\begin{equation}
 \label{best-constant-cone}
\norma{\widehat{f\sigma}}_{L^6(\R^3)}\leq (2\pi)^{5/6}\norma{f}_{L^2(\Gamma^2)}.
\end{equation}
and equality occurs in \eqref{best-constant-cone} if and only if $f(y,\ab{y})=e^{-a\ab{y}+b\cdot y+c}$, where $a,c\in\mathbb C$, 
$b\in\mathbb C^2$, and $\ab{\Re b}<\Re a$. 
\end{teo}
We will use this result to prove Theorem \ref{better-precompactness}.

Fanelli, Vega and Visciglia proved in \cite{FVV2} a general existence theorem for extremizers of Strichartz inequalities. We 
state here the case of the cone, in its equivalent form via \eqref{equivalent-form}. For $d\geq 2$ and $0\leq
\sigma<(d-1)/2$ the following Strichartz estimates hold (see \cite{FVV2}*{Example 1.1})
\begin{equation}
 \label{Strichartz-cone}
\norma{e^{it\sqrt{-\Delta}}g}_{L^{\frac{2(d+1)}{d-1-2\sigma}}_{t,x}(\R^{d+1})}\leq C\norma{g}_{\dot H^{\frac{1}{2}+\sigma}(\R^d)}.
\end{equation}

In \cite{FVV2}, using ``remodulation'' (equation after \cite{FVV2}*{equation 2.12}) ``rescaling'' and ``translation'' 
(\cite{FVV2}*{equation 2.15}) the following theorem is proved:

\begin{teo}\cite{FVV2}.
\label{from-fvv}
 Let $d\geq 2$ and $0<\sigma<(d-1)/2$. Then there exists an extremizer for \eqref{Strichartz-cone}. Moreover, extremizing 
 sequences are precompact, after the application of symmetries: ``remodulation'', ``rescaling'' and ``translation''.
\end{teo}

We point out here that their method does not apply to the endpoint case studied in this paper, $\sigma=0$ and $d=2$, because of 
the existence of further symmetries, Lorentz invariance, as discussed in Section \ref{lorentz-invariance}. The symmetries 
referred to in Theorem \ref{from-fvv}, when expressed in the dual formulation for $f\in L^2(\Gamma^2)$ are, in respective order:
\begin{itemize}
 \item $f(y)\rightsquigarrow e^{is\ab{y}}f(y)$, $s\in\R$,
 \item $f(y)\rightsquigarrow \la^{1/2}f(\la y)$, $\la>0$, and
 \item $f(y)\rightsquigarrow e^{iy\cdot y_0}f(y)$, $y_0\in\R^2$.
\end{itemize}
From the Lorentz invariance of inequality \eqref{restriction-cone-with-c}, and the fact that the Lorentz group is not generated
modulo a compact subgroup by the elements listed above, it follows that the final conclusion of Theorem \ref{from-fvv} cannot be
true in the endpoint case $d=2,\,\sigma=0$. This indicates that the proof in \cite{FVV2} likewise cannot apply to this endpoint
case.

On the one hand, for $d\geq 2$, under admissibility conditions in $(p,q)$ one has the Strichartz estimates
\[\norma{e^{it\sqrt{-\Delta}}g}_{L_t^pL_x^q(\R^{d+1})}\leq C\norma{g}_{\dot H^{\frac{1}{p}-\frac{1}{q}+\frac{1}{2}}(\R^d)},\]
so that for the case of the $\dot H^{1/2}(\R^d)$ one needs $p=q$ which then makes \cite{FVV2}*{Theorem 1.1} not 
applicable.

On the other hand, at the level of the proof of \cite{FVV2}*{Theorem 1.1}, one sees that \cite{FVV2}*{equation 2.12} does not hold
for $\sigma=0$ (or $s=\frac{1}{2}$ as appears there) and $d=2$. For this we show that there are extremizing sequences
$\{g_n\}_{n\in\N}$
such that $\norma{e^{it\sqrt{-\Delta}}g_n}_{L^\infty_t L_x^4}\to 0$ as $n\to\infty$. 

This is the same as having extremizing sequence $\{f_n\}_{n\in\N}$ such that $\norma{Tf_n}_{L^\infty_t L_x^4}\to 0$ as 
$n\to\infty$. For this we use the Lorentz invariance and the characterization of extremizers for the cone given in Theorem
\ref{extremizers-cone}.

From Section \ref{lorentz-invariance}, $\norma{T(f\circ L)}_{L^6(\R^3)}=\norma{Tf}_{L^6(\R^3)}$, for every Lorentz
transformation $L$ preserving $\Gamma^2$. Let $f$ be an $L^2$-normalized extremizer, say $f(x_1,x_2,x_3)=c_0e^{-x_3}$. We take
a sequence of Lorentz transformations $L^s$ and $f\circ L^s$ is also an $L^2$-normalized extremizer. We now compute
$\norma{T(f\circ L^s)}_{L^\infty_t L_x^4}$. We have
\[Tf(x,t)=\frac{2\pi c_0}{\sqrt{(1-it)^2+\ab{x}^2}}\]
and
\[\ab{Tf(x,t)}^4=\frac{(2\pi)^4c_0^4}{(1-t^2+\ab{x}^2)^2+4t^2}.\]
Now we use $L^s(x,t)=((x_1+st)/(1-s^2)^{1/2},x_2,(t+sx_1)/(1-s^2)^{1/2})$ and note that by making the change of 
variables $u=(x_1+st)(1-s^2)^{-1/2},\,v=x_2$ we obtain
\begin{align*}
 &\int_{\R^2} \ab{(Tf)\circ L^s(x,t)}^4dx=(1-s^2)^{1/2}\int_{\R^2} \ab{Tf(x_1,x_2,sx_1+t(1-s^2)^{1/2})}^4dx.
\end{align*}
Then, if $s\neq 0$
\begin{align*}
 &\sup\limits_{t\in\R}\int_{\R^2} \ab{T(f\circ L^s(x,t))}^4dx
   =(1-s^2)^{1/2}\sup\limits_{t\in\R}\int_{\R^2} \ab{Tf(x_1,x_2,s(x_1+t))}^4dx\\
  &=(2\pi)^4c_0^4(1-s^2)^{1/2}\sup\limits_{t\in\R}\int_{\R^2} \frac{dx_1dx_2}{(1-s^2(x_1+t)^2+x_1^2+x_2^2)^2+4s^2(x_1+t)^2}\\
  &=(2\pi)^4c_0^4(1-s^2)^{1/2}\sup\limits_{t\in\R}\int_{\R^2} \frac{dx_1dx_2}{(1-s^2x_1^2+(x_1+t)^2+x_2^2)^2+4s^2x_1^2}.
\end{align*}
It is not hard to show that for $(s,t)\in[\tfrac{1}{2},1]\times\R$
 \[\int_{\R^2} \frac{dx_1dx_2}{(1-s^2x_1^2+(x_1+t)^2+x_2^2)^2+4s^2x_1^2}\leq C,\]
with $C$ independent of $s$ and $t$. Therefore
\[\sup\limits_{t\in\R}\int_{\R^2} \ab{T(f\circ L^s(x,t))}^4dx\leq C(1-s^2)^{1/2}.\]
Hence $\lim_{s\to 1^-}\norma{T(f\circ L^s)}_{L^\infty_t L_x^4}=0$.

\vspace{0.3cm}
{\bf Notation:} We will write $X\lesssim Y$ or $Y\gtrsim X$ to denote an estimate of the form $X\leq CY$, and  $X\asymp Y$ to
denote an estimate of the form $cY\leq X\leq CY$, where $0<c,C<\infty$ are constants depending on fixed parameters of the problem,
but independent of $X$ and $Y$. We will denote by $\chi_A$ the characteristic function of a set $A$.

When writing integrals, we will sometimes drop the domain of integration or the measure when it is clear from context.

\section{The structure of the paper and the idea of the proof}

The proof of Theorem \ref{main-precompactness} follows the lines of the proof of precompactness of extremizing sequences for the 
adjoint Fourier restriction operator on the sphere $S^2\subset \R^3$ given in \cite{CS}.

In Section \ref{adjoint-fourier}, we give a known \cite{Str}, \cite{So}*{Chapter 2} proof of Theorem
\ref{adjoint-fourier-restriction}, with a view towards a refinement in terms of a cap space, as used in \cite{CS} and proved in
\cite{FVV} for compact surfaces in $\R^3$ of nonvanishing Gaussian curvature. In Section \ref{preliminaries}, we obtain bounds
that
we will use in Section \ref{section-cap-bound} to obtain the following cap estimate:
\begin{equation}
 \label{cap-estimate-intro}
 \norma{Tf}_{L^6(\R^3)}\lesssim \norma{f}_{L^2(\Gamma^2)}^{1-\gamma/2}
 \Bigl(\sup\limits_{\cp} \ab{\cp}^{-1/4}\int_{\cp}\ab{f}^{3/2}d\sigma\Bigr)^{\gamma/3},
\end{equation}
where the supremum ranges over all ``caps'' $\cp\subset\Gamma^2$ and $\gamma>0$ is a small universal constant. This is the 
analog of \cite{CS}*{Lemma 6.1}.

For a function satisfying $\norma{Tf}_{L^6(\R^3)}\geq \delta\mathbf{C}\norma{f}_{L^2(\Gamma^2)}$, the estimate in 
\eqref{cap-estimate-intro} allows the extraction of a cap $\cp$ with good properties: $f$ can be decomposed as the sum of two
functions with disjoint supports $f=g+h$ and $g$, which is comparable to $f\chi_\cp$, satisfies
\[\ab{g(x)}\leq C_\delta\norma{f}_2\ab{\cp}^{-1/2}\chi_{\cp}(x)\quad\text{ and }\quad\norma{g}_2\geq
\eta_\delta\norma{f}_2.\]

This is the content of Section \ref{using-cap-bound}. In Section \ref{lorentz-invariance}, we discuss symmetries of the cone. 
This includes dilations and Lorentz transformations and they allow us to take a cap $\cp$ and transform it into a cap $\cp'$ with
better properties: $\cp'$ is contained in a bounded region, independent of the extremizing sequence, and has big measure. 

The existence of symmetries of $(\Gamma^2,\sigma)$ simplifies the argument, compared to \cite{CS}. Two ways are possible, use the
arguments of Fanelli, Vega and Visciglia contained in \cite{FVV,FVV2} carried out in Section \ref{fvv-argument};
or use
a decomposition algorithm as done by Christ and Shao and carried out at the end of Section \ref{fvv-argument}.

For the argument based on \cite{FVV,FVV2}, a single extraction of a cap and the use of symmetries is enough to prove 
precompactness. In the case of the argument based on \cite{CS}, a cap decomposition is needed. For an extremizing sequence, the
cap decomposition is used to show that after dilations and Lorentz transformations, the extremizing sequence has a uniform
$L^2$-decay at infinity. The uniform decay plus a result inspired from \cite{FVV} allow us to complete the proof of
precompactness.

In the last section, we prove that extremizing sequences converge, after the application of symmetries of the cone. This is an 
easy task, that follows from the fact that the extremizers for \eqref{restriction-cone-with-c} are known and that the group of
symmetries of the cone acts transitively in the set of extremizers.

The author became recently aware that Ramos \cite{Ramos} proved a similar result to the cap estimate \eqref{cap-estimate-intro},
for all $d\geq 2$, using a different method that relies on bilinear estimates for the cone as proved by Wolff and Tao. He also
proved the existence of extremals for the Strichartz inequalities for the wave equation in the endpoint case.

\section{The adjoint Fourier restriction inequality}\label{adjoint-fourier}

Abusing notation, we will write $f(r,\te)=f(x)$, where $x=(r\cos\te,r\sin\te)$, which is the polar representation of $x$. Note 
that in polar coordinates the measure $\ab{y}^{-1}dy$ becomes $dr\,d\te$.

In the proof of Theorem \ref{adjoint-fourier-restriction}, we will need the following standard lemma.
\begin{lemma}[Fractional integration]
 \label{fractional-integration}
Let $p,q\in(1,\infty)$. Then for any $g\in L^p(\R),\, h\in L^q(\R)$ the following holds:
\[\int_\R\int_\R \ab{g(s)h(t)}\ab{t-s}^{-\alpha}dsdt\leq C_{p,q}\norma{g}_{L^p}\norma{h}_{L^q},\]
where $\alpha=2-1/p-1/q$ and $1/p+1/q>1$.
\end{lemma}
From Lemma \ref{fractional-integration} we have the following:
\begin{lemma}
 \label{fractional-integration-circle}
Let $p,q\in(1,\infty)$ with $1/p+1/q>1$ and let $\alpha=2-1/p-1/q$. Then for any 
$g\in L^p([0,2\pi]),\, h\in L^q([0,2\pi])$ the following holds:
\begin{equation}
 \label{bound-fractional-integration-circle}
\int_0^{2\pi}\int_0^{2\pi} \ab{g(s)h(t)}\ab{\sin(t-s)}^{-\alpha}dsdt\leq C_{p,q}\norma{g}_{L^p}\norma{h}_{L^q}.
\end{equation}
\end{lemma}
\begin{proof}
 We split the integral in sixteen pieces according to $[0,2\pi]=[0,\pi/2]\cup[\pi/2,\pi]\cup[\pi,3\pi/2]\cup[3\pi/2,2\pi]$, 
 and then it will be enough to show that
\[\int_{m\pi/2}^{(m+1)\pi/2}\int_{n\pi/2}^{(n+1)\pi/2} \ab{g(s)h(t)}\ab{\sin(t-s)}^{-\alpha}dsdt\leq 
  C_{p,q}\norma{g}_{L^p}\norma{h}_{L^q},\]
for all $m,n\in\{0,1,2,3\}$.
For this we use a simple change of variable that allows us to use Lemma \ref{fractional-integration}.

If $t,s\in[j\pi/2,(j+1)\pi/2]$, for some $j\in\{0,1,2,3\}$, then $\ab{t-s}\leq \pi/2$ and we use that $(2/\pi)\ab{t-s}\leq 
\ab{\sin(t-s)}\leq \ab{t-s}$.

If $s\in [0,\pi/2]$ and $t\in[\pi,3\pi/2]$ we can use the change of variables $t'=t-\pi$ so that $t'\in[0,\pi/2]$. We note that 
$\ab{\sin(t-s)}=\ab{\sin(t'-s)}$.

If $s\in[0,\pi/2]$ and $t\in[\pi/2,\pi]$ we further split the intervals as $[0,\pi/2]=[0,\pi/4]\cup[\pi/4,\pi/2]$ and 
$[\pi/2,\pi]=[\pi/2,3\pi/4]\cup[3\pi/4,\pi]$. If $s\in[0,\pi/4]$ and $t\in[\pi/2,3\pi/4]$ or if $s\in [\pi/4,\pi/2]$ and
$t\in[3\pi/4,\pi]$, then $\ab{\sin(t-s)}\geq 1/\sqrt{2}$ and the desired inequality follows from an application of H\"older's
inequality. If $s\in[\pi/4,\pi/2]$ and $t\in[\pi/2,3\pi/4]$, then $\ab{t-s}\leq \pi/2$ and we can use the inequality
$(2/\pi)\ab{t-s}\leq \ab{\sin(t-s)}\leq \ab{t-s}$ as in the first case discussed. Finally, if $s\in [0,\pi/4]$ and
$t\in[3\pi/4,\pi]$, we use the substitution $t'=t-\pi$ so that $t'\in[-\pi/4,0]$. Since $\ab{\sin(t-s)}=\ab{\sin(t'-s)}$ and
$\ab{t'-s}\leq \pi/2$ we can conclude as before.

The other cases follow in the same way.
\end{proof}

\begin{proof}[Proof of Theorem \ref{adjoint-fourier-restriction}]
We split $f(y)=\sum_{k\in\Z}^{}f_k(y)$ where $f_k(y)=f(y)\chi_{\{2^{k-1}\leq \ab{y}< 2^{k}\}}$. Then
\[(Tf)^2(x,t)=\suma{k,k'\in\Z}{}Tf_k\cdot Tf_{k'}.\]
Taking $L^3$ norm in both sides, using the triangle inequality and Lemma \ref{bilinear-estimate}, we obtain
\[\norma{Tf}^2_{L^6}\leq C\suma{k,k'}{}2^{-\ab{k-k'}/6}\norma{f_k}_{L^2(\sigma)}\norma{f_{k'}}_{L^2(\sigma)}.\]
To conclude, we use the Cauchy-Schwarz inequality
\begin{align*}
 \norma{Tf}^2_{L^6}&\leq C\biggl(\suma{k,k'}{}2^{-\ab{k-k'}/6}\norma{f_k}^2_{L^2(\sigma)}\biggr)^{1/2}
 \biggl(\suma{k,k'}{}2^{-\ab{k-k'}/6}\norma{f_{k'}}^2_{L^2(\sigma)}\biggr)^{1/2}\leq C \norma{f}^2_{L^2(\sigma)}.\qedhere
\end{align*}
\end{proof}

\begin{lemma}
\label{bilinear-estimate}
 There exists a constant $C<\infty$ with the following property. Let $k,k'\in\Z$ and $f,g\in L^2(\Gamma^2)$ with $f$ and $g$
supported in the regions $2^{k-1}\leq \ab{y}< 2^k$ and  $2^{k'-1}\leq \ab{y}< 2^{k'}$ respectively, then
\begin{equation}
\label{bilinear-estimate-bound}
 \norma{Tf\cdot Tg}_{L^3(\R^3)}\leq C 2^{-\ab{k-k'}/6}\norma{f}_{L^2(\Gamma^2)}\norma{g}_{L^2(\Gamma^2)}.
\end{equation}
\end{lemma}

\begin{proof}
We can split 
\[f(r,\te)g(r',\te')=f(r,\te)g(r',\te')(\chi_{r>r'}+\chi_{r<r'})(\chi_{\te>\te'}+\chi_{\te<\te'})\text{ for a.e }
  (r,r',\te,\te').\]
Thus, by the triangle inequality we can assume, without loss of generality, that $\te>\te'$ and $r<r'$ in the support of 
$f(r,\te)g(r',\te')$.

Using polar coordinates and Fubini's Theorem, we have
\begin{align*}
 Tf\cdot Tg(x,t)&=\int_{\R^2}\int_{\R^2}e^{ix\cdot (y+y')}e^{it(\ab{y}+\ab{y'})}f(y)g(y')\ab{y}^{-1}\ab{y'}^{-1}\;dydy'\\
&=\int e^{ix\cdot (r\cos\te+r'\cos\te',r\sin\te +r'\sin\te')}e^{it(r+r')}f(r,\te)g(r',\te')\;d\te d\te' drdr'.
\end{align*}

We make the following change of variables:
\[(r,r',\te,\te')\mapsto(u,s,\varrho)=(r\cos\te+r'\cos\te',r\sin\te +r'\sin\te',r+r',r),\]
which is injective in the region where $\te>\te'$, $r<r'$.
The Jacobian of the transformation is 
\[J^{-1}=\frac{\partial(u,s,\varrho)}{\partial(r,r',\te,\te')}=rr'\sin(\te-\te').\]

Using the change of variables

\[Tf\cdot Tg(x,t)=\int_0^\infty\Bigl(\int e^{ix\cdot u}e^{its}f(y)g(y')J duds\Bigr)d\varrho,\]
and by Minkowski's inequality and Hausdorff-Young inequality,
\begin{align*}
 \norma{Tf\cdot Tg}_{L^3}&\leq\int_0^\infty \Norma{\int e^{ix\cdot u}e^{its}f(y)g(y')J duds}_{L^3}d\varrho\\
&\leq C\int_0^\infty \Bigl(\int \ab{f(y)g(y')J}^{3/2} duds\Bigr)^{2/3}d\varrho\\
&=C\int_0^\infty\Bigl(\int\ab{f(y)g(y')}^{3/2}(rr')^{-\frac{1}{2}}\ab{\sin(\te-\te')}^{-\frac{1}{2}}Jduds\Bigr)^{2/3}d\varrho.
\end{align*}
We now use that $r\asymp 2^{k}$, $r'\asymp 2^{k'}$ and H\"older's inequality to obtain
\begin{align}
 \norma{Tf\cdot Tg}_{L^3}&\leq C (2^k 2^{k'})^{-1/3}(2^{k})^{1/3}\Bigl(\int
\ab{f(y)g(y')}^{3/2}\ab{\sin(\te-\te')}^{-\frac{1}{2}}Jdudsd\varrho\Bigr)^{2/3}\nonumber\\
 \label{bilinear-with-sine}
&=C (2^k 2^{k'})^{-1/3}(2^{k})^{1/3}\Bigl(\int \ab{f(y)g(y')}^{3/2}\ab{\sin(\te-\te')}^{-\frac{1}{2}}d\te d\te'drdr'\Bigr)^{2/3}.
\end{align}
On the other hand, by Lemma \ref{fractional-integration-circle},
\begin{align*}
 \int &\ab{f(y)g(y')}^{3/2}\ab{\sin(\te-\te')}^{-\frac{1}{2}}d\te d\te'drdr' \\
&\leq C \int \Bigl(\int \ab{f(r,\te)}^2 d\te\Bigr)^{3/4}dr\cdot \int \Bigl(\int \ab{g(r,\te')}^2 d\te'\Bigr)^{3/4}dr'\\
&\leq C (2^k2^{k'})^{1/4}\Bigl(\int \ab{f(r,\te)}^2 drd\te\Bigr)^{3/4}\Bigl(\int \ab{g(r',\te')}^2dr'd\te'\Bigr)^{3/4}.
\end{align*}
Then, as $2^k\leq 2^{k'}$
\begin{align*}
 \norma{Tf\cdot Tg}_{L^3}&\leq C(2^k
2^{k'})^{-1/3}\min((2^{k})^{1/3},(2^{k'})^{1/3})(2^k2^{k'})^{1/6}\norma{f}_{L^2_{r,\te}}\norma{g}_{L^2_{r,\te}}\\
&=C 2^{-(k+k')/6}\min((2^{k})^{1/3},(2^{k'})^{1/3})\norma{f}_{L^2(\sigma)}\norma{g}_{L^2(\sigma)}.
\end{align*}

We note that $2^{-(k+k')/6}\min((2^{k})^{1/3},(2^{k'})^{1/3})=2^{-\ab{k-k'}/6}$, so
\[\norma{Tf\cdot Tg}_{L^3}\leq C2^{-\ab{k-k'}/6}\norma{f}_{L^2(\sigma)}\norma{g}_{L^2(\sigma)}.\qedhere\]
\end{proof}

\begin{prop}
\label{improved-exponent}
 There exists a constant $C<\infty$ with the following property. Let $f\in L^2(\Gamma^2)$ and for $k\in \Z$ let
$f_k(y)=f(y)\chi_{\{2^{k-1}\leq \ab{y}<2^k\}}$. Then
 \[\norma{Tf}_{L^6(\R^3)}\leq C\Bigl(\suma{k\in\Z}{}\norma{f_k}_{L^2(\Gamma^2)}^3\Bigr)^{1/3}.\]
\end{prop}
\begin{proof}
 By rewriting $\norma{Tf}_{L^6(\R^3)}^3$ as the $L^2$ norm of a trilinear form and using the triangle inequality, we have
 \begin{align*}
  \norma{Tf}_{L^6}^3&=\norma{Tf\cdot Tf\cdot Tf}_{L^2}=\biggl\Vert\suma{i,j,k}{}Tf_i\cdot Tf_j\cdot Tf_k\biggr\Vert_{L^2}\leq
\suma{i,j,k}{}\norma{Tf_i\cdot Tf_j\cdot Tf_k}_{L^2}.
 \end{align*}
 Now for each $i,j$ and $k$, without loss of generality, we can assume that $\ab{j-k}=\max(\ab{i'-j'}:i',j'\in\{i,j,k\})$. Using 
 H\"older's inequality, Theorem \ref{adjoint-fourier-restriction} and Lemma \ref{bilinear-estimate}, we obtain
 \begin{equation}
 \label{ijk}
 \norma{Tf_i\cdot Tf_j\cdot Tf_k}_{L^2}\leq \norma{Tf_i}_{L^6}\norma{Tf_j\cdot Tf_k}_{L^3}\leq 
        C2^{-\ab{j-k}/6}\norma{f_i}_{L^2}\norma{f_j}_{L^2}\norma{f_k}_{L^2}.
 \end{equation}
 Now, using the maximality of $\ab{j-k}$ we see that $\ab{j-k}\geq \frac{1}{3}\ab{i-j}+\frac{1}{3}\ab{j-k}+\frac{1}{3}\ab{k-i}$, 
 and hence from \eqref{ijk},
 \[\norma{Tf_i\cdot Tf_j\cdot Tf_k}_{L^2}\leq 2^{-\ab{i-j}/18}2^{-\ab{j-k}/18}2^{-\ab{k-i}/18}
   \norma{f_i}_{L^2}\norma{f_j}_{L^2}\norma{f_k}_{L^2}.\]
 Then
 \[\norma{Tf}_{L^6}^3\leq C\suma{i,j,k}{}2^{-\ab{i-j}/18}2^{-\ab{j-k}/18}2^{-\ab{k-i}/18}\norma{f_i}_{L^2}\norma{f_j}_{L^2}
   \norma{f_k}_{L^2},\]
 and a final application of H\"older's inequality gives the desired conclusion
 \begin{align*}
  \norma{Tf}_{L^6}^3&\leq C\suma{i,j,k}{}2^{-\ab{i-j}/18}2^{-\ab{j-k}/18}2^{-\ab{k-i}/18}\norma{f_k}_{L^2}^3\leq 
   C\suma{k\in\Z}{}\norma{f_k}_{L^2}^3.\qedhere
 \end{align*}
\end{proof}

\section{Preliminaries for the cap bound for the adjoint Fourier operator}\label{preliminaries}

Recall that in the computation of $\norma{(Tf)^2}_{L^3}$, in equation \eqref{bilinear-with-sine} with $g=f$, we came across 
the expression
\[\int \ab{f(r,\te)f(r',\te')}^{3/2}\ab{\sin(\te-\te')}^{-1/2}d\te d\te' drdr'.\]

By assuming the angular support of $f$ is contained in the region $0\leq \te\leq \pi/2$, that is, $f(r,\te)=0$ if 
$\te\notin[0,\pi/2]$, we can study instead the comparable expression
\[\int \ab{f(r,\te)f(r',\te')}^{3/2}\ab{\te-\te'}^{-1/2}d\te d\te' drdr'.\]

Instead of using fractional integration in $\te$ and $\te'$ and H\"older's inequality in $r$ and $r'$ we want to obtain a ``cap
type'' 
inequality for $T$ of the form in \cite{MVV}*{Theorem 4.2}.

\begin{define}
\label{definition-of-cap}
 By a cap $\cp$, we mean a set $\cp\subset \Gamma$ whose projection to the plane $\R^2\times\{0\}$ is of the form 
 $[2^{k-1},2^k]\times J$, when written in polar coordinates $(r,\te)$, where $k\in\Z$ and $J\subset [0,2\pi]$ is an interval. We
will identify the cap $\cp$ with its projection to the plane $\R^2\times\{0\}$ and write $\cp=[2^{k-1},2^k]\times J$.
\end{define}

For a cap $\cp=[2^{k-1},2^k]\times J$, $\ab{\cp}:=\sigma(\cp)=2^{k-1}\ab{J}$, and for any $\la\geq 0$, 
$\la\cp=[\la 2^{k-1},\la2^k]\times J$, so $\sigma(\la \cp)=\la\sigma(\cp)$.

\begin{define}
 Let $0< \alpha<1$ and $p=2/(2-\alpha)$. Define, for $f,g\in L^p(\R)$, the bilinear operator
 \begin{equation}
 \label{B-operator}
 B(f,g)=\int_{\R^2} f(x)g(x')\ab{x-x'}^{-\alpha}dxdx'.
 \end{equation}
\end{define}
Note that the kernel $x\in\R\mapsto\ab{x}^{-\alpha}$ has a strictly positive Fourier transform, and thus $B$ is nondegenerate 
and satisfies the Cauchy-Schwarz inequality $\ab{B(f,g)}^2\leq B(f,f)B(g,g)$.

Lemma \ref{fractional-integration} implies that $\ab{B(f,f)}\leq C_p\norma{f}_{L^p(\R)}^2$. We can say more if we work with the 
Lorentz spaces $L^{p,q}(\R)$ (see \cite{SW} for an introduction to Lorentz spaces). We have the following bound for $B$
\cite{ONeil}:
\[\ab{B(f,f)}\lesssim \norma{f}_{L^{p,2}(\R)}^2.\]
This bound will allow us to prove the following:
\begin{prop}
\label{cap-bound-for-b-splitted-f}
 Let $0<\alpha<1$ and $p=2/(2-\alpha)$. There exist constants $C<\infty$ and $\delta\in(0,2)$ such that for all $f\in L^p(\R)$
the following inequality holds:
 \[B(f,f)\leq C \norma{f}_{L^p(\R)}^{2-\delta}\sup\limits_{k,I}\norma{f_k}_{L^p(\R)}^{\delta}
   \biggl(\frac{\ab{E_k\cap I}}{\ab{E_k}+\ab{I}}\biggr)^{\delta},\]
 where $I$ ranges over all compact intervals of $\R$, $E_k=\{x\in\R:2^k\leq \ab{f(x)}<2^{k+1}\}$ and $f_k=f\chi_{E_k}$, $k\in\Z$.
\end{prop}

\begin{proof}
We will use the following characterization of the $L^{p,2}$ norm. If we decompose $f$ as in the statement of the proposition, 
$f=\sum_{k\in\Z}{}f_k$ where $f_k$ have disjoint supports, $E_k$, and $2^{k}\chi_{E_k}\leq \ab{f_k}<2^{k+1}\chi_{E_k}$, then
\begin{equation}
 \label{lp2-norm-decomposed}
 \norma{f}_{L^{p,2}(\R)}^2\asymp \suma{k\in\Z}{}\norma{f_k}_{L^p(\R)}^2.
\end{equation}

It follows from \eqref{lp2-norm-decomposed} that $\norma{f}_{L^{p,2}}^2\lesssim\norma{f}_{L^p}^p\sup_k\norma{f_k}_{L^p}^{2-p}$, 
from where the following bound is obtained:
\[\ab{B(f,f)}\lesssim\norma{f}_{L^p(\R)}^p\sup\limits_{k\in\Z}\norma{f_k}_{L^p(\R)}^{2-p}. \]

We can improve the previous estimate. For this, let $\eta>0$, $S=\{k: \norma{f_k}_p\geq \eta\norma{f}_p\}$, and 
$g=\sum_{k\in S}{}f_k$. Then $\ab{B(f-g,f-g)}\lesssim \eta^{2-p}\norma{f}_{L^p}^2$. Since
$\norma{f}_{L^p}^p=\sum_{k}\norma{f_k}_{L^p}^p$ we obtain that $\ab{S}\leq \eta^{-p}$. Therefore, by Cauchy-Schwarz
\[\ab{B(g,g)}^{1/2}\leq\suma{k\in S}{}\ab{B(f_k,f_k)}^{1/2}\leq \ab{S}\max\limits_{k\in S}
  \ab{B(f_k,f_k)}^{1/2}\leq \eta^{-p}\max\limits_{k\in S}\ab{B(f_k,f_k)}^{1/2}.\]

We deduce that 
\[\ab{B(f,f)}^{1/2}\leq \ab{B(f-g,f-g)}^{1/2}+\ab{B(g,g)}^{1/2}\leq  
  \eta^{(2-p)/2}\norma{f}_{L^p}+\eta^{-p}\max\limits_{k\in S}\ab{B(f_k,f_k)}^{1/2},\]
and squaring we obtain that for all $\eta>0$
\[\ab{B(f,f)}\lesssim \eta^{2-p}\norma{f}_{L^p}^{2}+\eta^{-2p}\sup\limits_{k\in\Z}\ab{B(f_k,f_k)}.\]
Optimizing in $\eta$ gives
\begin{equation}
 \label{optimazed-in-eta}
 \ab{B(f,f)}\lesssim \sup\limits_{k\in\Z}\ab{B(f_k,f_k)}^{\delta/2}\norma{f}_{L^p}^{2-\delta},
\end{equation}
for some $\delta\in(0,1)$ (the optimization gives $\delta=2(2-p)/(2+p)$). Thus, it is then enough to obtain a bound on 
$B(f,f)$ where $f=\chi_E$.

\begin{lemma}
 \label{b-onsimple-functions}
 There exist $C<\infty$ and $\gamma\in(0,1)$ with the following property. For every $E$ subset of $\R$ of finite measure
 \begin{equation}
  \label{estimate-b-simple-function}
  B(\chi_E,\chi_E)\leq C\norma{\chi_E}_{L^p(\R)}^{2}\biggl(\sup\limits_{I}\frac{\ab{E\cap I}}{\ab{E}+\ab{I}}\biggr)^\gamma,
 \end{equation}
where the supremum ranges over all compact intervals $I$ of $\R$.
\end{lemma}
\begin{proof}
 Let $\{I_j^k\}_{j\in\Z}$ be a partition of the real line into intervals of equal length $2^{k}$. Then

\begin{align*}
 B(\chi_E,\chi_E)&=\iint\frac{\chi_E(x)\chi_E(y)}{\ab{x-y}^\alpha}\,dxdy=
 \suma{k}{}\iint\limits_{\{2^{k-1}\leq \ab{x-y}<2^k\} }\frac{\chi_E(x)\chi_E(y)}{\ab{x-y}^\alpha}\,dxdy\\
&\asymp \suma{k}{}\suma{j}{}2^{-k\alpha}\ab{E\cap I_j^k}\ab{E\cap \tilde I_j^k}\\
&\lesssim \suma{k}{}\suma{j}{}2^{-k\alpha}\ab{E\cap \tilde I_j^k}^2
\end{align*}
where $\tilde I_j^k$ has the same center as $I_j^k$ and double length. From now on we will rename $\tilde I_j^k$ by $I_j^k$.
                                                                                                                                                                                                                                                                                          
Now, we fix $k$ and estimate $\sum_{j}{}2^{-k\alpha}\ab{E\cap  I_j^k}^2$. Let $n$ be such that $2^n\leq \ab{E}<2^{n+1}$. We 
will divide the analysis into the cases where $k\leq n$ and $k>n$. Recall that $p=2/(2-\alpha)$, and let $\gamma\in (0,1)$ be a
number to be determined later. We first consider the case $k\leq n$. We have
\begin{align*}
 \suma{j}{}2^{-k\alpha}\ab{E\cap  I_j^k}^2&\leq \suma{j}{}\ab{E\cap I_j^k} 2^{-k\alpha}\sup\limits_{i}\ab{E\cap I_i^k}\\
&\lesssim \ab{E}2^{-k\alpha}\Bigl(\sup\limits_{i}\frac{\ab{E\cap I_i^k}}{\ab{E}+\ab{I_i^k}}\Bigr)^\gamma 2^{k(1-\gamma)}
   \ab{E}^\gamma\\
&\asymp \ab{E}^{1+\gamma}2^{k(1-\alpha-\gamma)}\Bigl(\sup\limits_{i}\frac{\ab{E\cap I_i^k}}{\ab{E}+\ab{I_i^k}}\Bigr)^\gamma\\
&=\ab{E}^{2-\alpha}\ab{E}^{-1+\alpha+\gamma}2^{k(1-\alpha-\gamma)}
  \Bigl(\sup\limits_{i}\frac{\ab{E\cap I_i^k}}{\ab{E}+\ab{I_i^k}}\Bigr)^\gamma\\
&\lesssim\norma{\chi_E}_{L^p}^{2}2^{-(n-k)(1-\alpha-\gamma)}\Bigl(\sup\limits_{i}\frac{\ab{E\cap
  I_i^k}}{\ab{E}+\ab{I_i^k}}\Bigr)^\gamma.
\end{align*}
Now, if $k>n$ we have
\begin{align*}
 \suma{j}{}2^{-k\alpha}\ab{E\cap  I_j^k}^2&\leq \suma{j}{}\ab{E\cap I_j^k} 2^{-k\alpha}\sup\limits_{i}\ab{E\cap I_i^k}\\
&\lesssim \ab{E}2^{-k\alpha}\Bigl(\sup\limits_{i}\frac{\ab{E\cap I_i^k}}{\ab{E}+\ab{I_i^k}}\Bigr)^\gamma 2^{k\gamma}
  \ab{E}^{1-\gamma}\\
&\asymp \ab{E}^{2-\gamma}2^{-k(\alpha-\gamma)}\Bigl(\sup\limits_{i}\frac{\ab{E\cap I_i^k}}{\ab{E}+\ab{I_i^k}}\Bigr)^\gamma\\
&=\ab{E}^{2-\alpha}\ab{E}^{\alpha-\gamma}2^{-k(\alpha-\gamma)}
  \Bigl(\sup\limits_{i}\frac{\ab{E\cap I_i^k}}{\ab{E}+\ab{I_i^k}}\Bigr)^\gamma\\
&\lesssim\norma{\chi_E}_{L^p}^{2}2^{-(k-n)(\alpha-\gamma)}\Bigl(\sup\limits_{i}\frac{\ab{E\cap
  I_i^k}}{\ab{E}+\ab{I_i^k}}\Bigr)^\gamma.
\end{align*}
Thus, if we choose $\gamma>0$ smaller than $\min(1-\alpha,\alpha)$, we obtain the desired conclusion after adding over $k$
\[B(\chi_E,\chi_E)\lesssim\norma{\chi_E}_{L^p}^{2}\biggl(\sup\limits_{I}\frac{\ab{E\cap
I}}{\ab{E}+\ab{I}}\biggr)^\gamma.\qedhere\]
\end{proof}
By combining Lemma \ref{b-onsimple-functions} and \eqref{optimazed-in-eta} we obtain that for $f\in L^p$
\[B(f,f)\lesssim \norma{f}_{L^p}^{2-\delta}\sup\limits_{k,I}\norma{f_k}_{L^p}^{\delta}
  \biggl(\frac{\ab{E_k\cap I}}{\ab{E_k}+\ab{I}}\biggr)^{\delta\gamma/2},\]
that implies (after we rename $\delta\gamma/2$ by $\delta$)
\[B(f,f)\lesssim \norma{f}_{L^p}^{2-\delta}\sup\limits_{k,I}\norma{f_k}_{L^p}^{\delta}
  \biggl(\frac{\ab{E_k\cap I}}{\ab{E_k}+\ab{I}}\biggr)^{\delta},\]
since $\norma{f_k}_p/\norma{f}_p\leq 1$ and so $(\norma{f_k}_p/\norma{f}_p)^\delta\leq 
(\norma{f_k}_p/\norma{f}_p)^{\delta\gamma/2}$.
\end{proof}

We note that $\sup_{k,I}\norma{f_k}_{L^p}^{\delta}(\ab{E_k\cap I}/(\ab{E_k}+\ab{I}))^{\delta}$ is bounded by
$\bigl(\sup_I \ab{I}^{-1+1/p}\int_I \ab{f}\bigr)^\delta$. Indeed, we have 
\[\sup\limits_{k,I}\norma{f_k}_{L^p}^{\delta}\biggl(\frac{\ab{E_k\cap I}}{\ab{E_k}+\ab{I}}\biggr)^{\delta}\lesssim
\biggl(\sup\limits_{k,I} \ab{I}^{-1+1/p}\int_I \ab{f_k}\biggr)^\delta.\]

To see this, we rewrite $\norma{f_k}_{L^p}^{\delta}\asymp 2^{k\delta}\ab{E_k}^{\delta/p}$ and $\int_I \ab{f_k}\asymp
2^k\ab{E_k\cap I}$. It suffices to show that for all $k$ and $I$
\[\ab{E_k}^{\delta/p}\biggl(\frac{\ab{E_k\cap I}}{\ab{E_k}+\ab{I}}\biggr)^\delta\leq \ab{I}^{(-1+1/p)\delta}\ab{E_k\cap
I}^\delta,\]
which is equivalent to $\ab{E_k}^{\delta/p}\ab{I}^\delta\leq \ab{I}^{\delta/p}(\ab{E_k}+\ab{I})^\delta$. This holds trivially 
in the case $\ab{E_k}\leq \ab{I}$, while in the case $\ab{E_k}>\ab{I}$, we rewrite the inequality as
\[1\leq \biggl(1+\frac{\ab{I}}{\ab{E_k}}\biggr)^\delta \biggl(\frac{\ab{I}}{\ab{E_k}}\biggr)^{\delta(-1+1/p)}, \]
which holds because $-1+1/p<0$.

We have proved the following proposition.
\begin{prop}
Let $0<\alpha<1$ and $p=2/(2-\alpha)$. There exist $C<\infty$ and $\delta\in(0,2)$ such that for all $f\in L^p(\R)$
 \begin{equation}
 \label{bound-b-l1}
B(f,f)\leq C\norma{f}_{L^p(\R)}^{2-\delta}\biggl(\sup\limits_{I} \ab{I}^{-1+1/p}\int_I \ab{f}dx\biggr)^{\delta}.
\end{equation}
where $I$ ranges over all compact intervals of $\R$.
\end{prop}
Using the Cauchy-Schwarz inequality for $B$ and a decomposition as in Lemma \ref{fractional-integration-circle}, we obtain 
the corollary as follows:
\begin{cor}
 Let $0<\alpha<1$ and $p=2/(2-\alpha)$. There exist $C<\infty$ and $\delta\in(0,2)$ such that for all $f\in L^p([0,2\pi])$,
\[\int_{[0,2\pi]^2} f(x)f(y)\ab{\sin(x-y)}^{-\alpha}dxdy\leq C\norma{f}_{L^p([0,2\pi])}^{2-\delta}\biggl(\sup\limits_{I}
\ab{I}^{-1+1/p}\int_I
\ab{f}dx\biggr)^{\delta},\]
where $I$ ranges over all subintervals of $[0,2\pi]$.
\end{cor}

We now consider the operator we will use to control the adjoint Fourier operator $T$.
\begin{define}
 Let $0<\alpha<1$ and $p=2/(2-\alpha)$. We define the bilinear operator $Q:L^p(\R^2)\times L^p(\R^2)\to\R$ by
 \begin{equation}
  \label{definition-of-q}
  Q(f,g)=\int_{(\R^2)^2} f(r,x)g(r',x')\ab{x-x'}^{-\alpha}dxdx'drdr',
 \end{equation}

\end{define}
Note that we can write $Q(f,f)=B(\int f(r,x)dr,\int f(r',x')dr')$.

For $f\in L^p_{r,x}$ with $\norma{\int f(r,x)dr}_{L^p_x}<\infty$, we use \eqref{bound-b-l1}  to obtain 
\begin{equation}
\label{relating-q-and-b}
Q(f,f)\lesssim \biggl\Vert\int \ab{f(r,x)}dr\biggr\Vert_{L^p_x}^{2-\delta}\biggl(\sup\limits_{I} \ab{I}^{-1+1/p}
  \int_I\int \ab{f(r,x)}drdx\biggr)^{\delta}.
\end{equation}

Suppose that $f(r,x)$ is supported where $2^{k-1}\leq r<2^k$, then $\int_I\int f(r,x)drdx=\int_\cp f(r,x)$, where
$\cp=[2^{k-1},2^k]\times I$, and $\norma{\int f(r,x)dr}_{L^p_x}\leq 2^{(k-1)(1-1/p)}\norma{f}_{L^p_{r,x}}$. Thus, it follows from
\eqref{relating-q-and-b} that
\begin{equation}
 \label{bound-for-q}
 Q(f,f)\lesssim 2^{2k(1-1/p)}\norma{f}_{L^p_{r,x}}^{2-\delta}\biggl(\sup\limits_{\cp} \ab{\cp}^{-1+1/p}\int_\cp
\ab{f(r,x)}drdx\biggr)^{\delta} ,
\end{equation}
where we used $2^{k-1}\ab{I}=\ab{\cp}$.

In the case we are interested in, we will need to estimate $Q(\ab{f_k}^{3/2},\ab{f_k}^{3/2})$ with the support of $f_k$ as before
and $f_k\in L^2_{r,x}$, with $\alpha=\frac{1}{2}$ and $p=\frac{4}{3}$.

\begin{cor}
 \label{cross-term-in-q}
There exist $C<\infty$ and $\delta\in(0,2)$ with the following property. Let $k,k'\in\Z$ and $f,g\in L^{4/3}(\R^2)$ and suppose
that
$f(r,x),g(r,x)$ are supported in the regions $[2^{k-1},2^k]\times \R$ and $[2^{k'-1},2^{k'}]\times \R$ respectively. Then
 \begin{align}
 \ab{Q(f,g)}^2\leq C &2^{(k+k')/2}\norma{f}_{L^{4/3}_{r,x}}^{2-\delta}\norma{g}_{L^{4/3}_{r,x}}^{2-\delta}
 \biggl(\sup\limits_{\cp} \ab{\cp}^{-1/4}\int_\cp \ab{f}drdx\biggr)^{\delta}\nonumber\\
 \label{bound-q-cross-term}
 &\cdot \biggl(\sup\limits_{\cp}\ab{\cp}^{-1/4}\int_\cp\ab{g}drdx\biggr)^{\delta}.
\end{align}
\end{cor}
\begin{proof}
 This follows from \eqref{bound-for-q} and the Cauchy-Schwarz inequality for $Q$,
 \[Q(f,g)^2\leq Q(f,f)Q(g,g).\qedhere\]
\end{proof}
For $f_k,f_{k'}\in L^2(\R^2_{(r,x)})$ supported where $2^{k-1}\leq r <2^k$ and $2^{k'-1}\leq r <2^{k'}$, respectively, we obtain
\begin{multline}
\label{bound-q-cross-term-3/2}
 Q(\ab{f_k}^{3/2},\ab{f_{k'}}^{3/2})^2\lesssim
2^{(k+k')/2}\norma{f_k}_{L^2_{r,x}}^{3(2-\delta)/2}\norma{f_{k'}}_{L^2_{r,x}}^{3(2-\delta)/2}\\
\cdot \Bigl(\sup\limits_{\cp} \ab{\cp}^{-1/4}\int_\cp \ab{f_k}^{3/2}drdx\Bigr)^{\delta}\Bigl(\sup\limits_{\cp}
\ab{\cp}^{-1/4}\int_\cp \ab{f_{k'}}^{3/2}drdx\Bigr)^{\delta}.
\end{multline}

The use of the Cauchy-Schwarz inequality for $Q$, and a decomposition as in Lemma \ref{fractional-integration-circle} implies that
for $f_k,f_{k'}\in L^2(\R_r\times [0,2\pi]_x)$ supported where $2^{k-1}\leq r <2^k$ and $2^{k'-1}\leq r <2^{k'}$ the following
estimate holds:
\begin{multline}
 \label{bound-Q-circle}
 \Bigl(\int\ab{f_k(r,x)f_{k'}(r',x')}^{3/2}\ab{\sin(x-x')}^{-1/2}dxdx'drdr'\Bigr)^2
\lesssim 2^{(k+k')/2}\norma{f_k}_{L^2_{r,x}}^{3(2-\delta)/2}\\
\cdot\norma{f_{k'}}_{L^2_{r,x}}^{3(2-\delta)/2} \Bigl(\sup\limits_{\cp} \ab{\cp}^{-1/4}\int_\cp
\ab{f_k}^{3/2}drdx\Bigr)^{\delta}\Bigl(\sup\limits_{\cp} \ab{\cp}^{-1/4}\int_\cp \ab{f_{k'}}^{3/2}drdx\Bigr)^{\delta}.
\end{multline}

\section{The cap bound for the adjoint Fourier restriction operator}\label{section-cap-bound}
\begin{prop}
 \label{cap-bilinear-estimate}
 There exist $C<\infty$ and $\delta\in(0,2)$ with the following property. Let $k,k'\in\Z$ and $f,g\in L^2(\Gamma^2)$, with $f$ 
 and $g$ supported in the regions $2^{k-1}\leq \ab{y}<2^k$ and $2^{k'-1}\leq\ab{y}<2^{k'}$, respectively, then
 \begin{multline}
  \label{cap-bilinear-inequality}
  \norma{Tf\cdot Tg}_{L^3(\R^3)}\leq C
2^{-\ab{k-k'}/6}\norma{f}_{L^2(\Gamma^2)}^{(2-\delta)/2}\norma{g}_{L^2(\Gamma^2)}^{(2-\delta)/2}\\
\cdot \Bigl(\sup\limits_{\cp} \ab{\cp}^{-1/4}\int_\cp \ab{f}^{3/2}d\sigma\Bigr)^{\delta/3}\Bigl(\sup\limits_{\cp}
\ab{\cp}^{-1/4}\int_\cp \ab{g}^{3/2}d\sigma\Bigr)^{\delta/3}.
 \end{multline}
\end{prop}
\begin{proof}
Recall from Section \ref{adjoint-fourier}, equation \eqref{bilinear-with-sine}, that we have the inequality
\begin{multline*}
 \norma{Tf\cdot Tg}_{L^3}\leq C (2^k 2^{k'})^{-1/3}\min(2^{k},2^{k'})^{1/3}\\
\cdot\Bigl(\int \ab{f(y)g(y')}^{3/2}\ab{\sin(\te-\te')}^{-1/2}d\te d\te'drdr'\Bigr)^{2/3}.
\end{multline*}

From \eqref{bound-Q-circle}, we obtain
\begin{multline*}
\norma{Tf\cdot Tg}_{L^3}\lesssim (2^k
2^{k'})^{-1/3}\min(2^{k},2^{k'})^{1/3}2^{(k+k')/6}\norma{f}_{L^2(\Gamma^2)}^{(2-\delta)/2}\norma{g}_{L^2(\Gamma^2)}^{(2-\delta)/2}
\\
\cdot\Bigl(\sup\limits_{\cp} \ab{\cp}^{-1/4}\int_\cp \ab{f}^{3/2}d\sigma\Bigr)^{\delta/3}\Bigl(\sup\limits_{\cp}
\ab{\cp}^{-1/4}\int_\cp
\ab{g}^{3/2}d\sigma\Bigr)^{\delta/3},
\end{multline*}
which as in the proof of Lemma \ref{bilinear-estimate} can be rewritten as
\begin{multline*}
\norma{Tf\cdot Tg}_{L^3}\lesssim
2^{-\ab{k-k'}/6}\norma{f}_{L^2(\Gamma^2)}^{(2-\delta)/2}\norma{g}_{L^2(\Gamma^2)}^{(2-\delta)/2}\\
\cdot\Bigl(\sup\limits_{\cp} \ab{\cp}^{-1/4}\int_\cp \ab{f}^{3/2}d\sigma\Bigr)^{\delta/3}\Bigl(\sup\limits_{\cp}
\ab{\cp}^{-1/4}\int_\cp
\ab{g}^{3/2}d\sigma\Bigr)^{\delta/3}.\qedhere
\end{multline*}
\end{proof}

\begin{cor}
\label{cap-decomposition-of-T}
 There exist $C<\infty$ and $\delta\in(0,2)$ with the following property. If $f\in L^2(\Gamma^2)$ and $f_k=f\chi_{\{2^{k-1}\leq
\ab{y}<2^k\}}$, $k\in\Z$, then
\begin{align}
\label{xp-cap-bound}
 \norma{Tf}^2_{L^6(\R^3)}&\leq C \suma{k\in\Z}{}\norma{f_k}_{L^2(\Gamma^2)}^{2-\delta}\Bigl(\sup\limits_{\cp}
\ab{\cp}^{-1/4}\int_\cp
\ab{f_k}^{3/2}d\sigma\Bigr)^{2\delta/3}.
\end{align}
\end{cor}
\begin{proof}
 We start by writing $\norma{Tf}^2_{L^6}=\norma{Tf\cdot Tf}_{L^3}$ and $Tf=\sum_{k\in\Z} Tf_k$, so the triangle inequality gives
 \[\norma{Tf\cdot Tf}_{L^3}\leq \suma{k,k'}{}\norma{Tf_k\cdot Tf_{k'}}_{L^3}\]
 that together with Proposition \ref{cap-bilinear-estimate} gives 
\begin{multline*}
\norma{Tf}^2_{L^6}\lesssim \suma{k,k'}{}
2^{-\ab{k-k'}/6}\norma{f_k}_{L^2(\Gamma^2)}^{(2-\delta)/2}\norma{f_{k'}}_{L^2(\Gamma^2)}^{(2-\delta)/2}\\
\cdot\Bigl(\sup\limits_{\cp} \ab{\cp}^{-1/4}\int_\cp \ab{f_k}^{3/2}d\sigma\Bigr)^{\delta/3}\Bigl(\sup\limits_{\cp}
\ab{\cp}^{-1/4}\int_\cp \ab{f_{k'}}^{3/2}d\sigma\Bigr)^{\delta/3}.
\end{multline*}
The desired conclusion follows by the Cauchy-Schwarz inequality.
\end{proof}
By using Proposition \ref{cap-bilinear-estimate} instead of Lemma \ref{bilinear-estimate} we can obtain an analog of Proposition 
\ref{improved-exponent}, which is as follows: 
\begin{prop}
\label{dyadic-cap-estimate-for-T}
There exist $C<\infty$ and $\delta\in(0,2)$ with the following property. Let $f\in L^2(\Gamma^2)$ and for $k\in\Z$ let
$f_k(y)=f(y)\chi_{\{2^{k-1}\leq \ab{y}<2^{k}\}}$. Then
 \begin{equation}
  \label{dyadic-cap-estimate-for-T-inequality}
  \norma{Tf}_{L^6(\R^3)}\leq C \Biggl(\suma{k\in\Z}{}\norma{f_k}_{L^2(\Gamma^2)}^{3-3\delta/2}
\biggl(\sup\limits_{\cp} \ab{\cp}^{-1/4}\int_\cp \ab{f_k}^{3/2}d\sigma\biggr)^{\delta}\Biggr)^{1/3}.
 \end{equation}
\end{prop}

\begin{prop}[Cap estimate]
 \label{cap-estimate-for-T}
There exist $C<\infty$ and $\delta\in(0,2)$ such that for all $f\in L^2(\Gamma^2)$ the following estimate holds:
\begin{equation}
 \label{cap-estimate-for-T-inequality}
 \norma{Tf}_{L^6(\R^3)}\leq C\norma{f}_{L^2(\Gamma^2)}^{1-\delta/2}\biggl(\sup\limits_{\cp} \ab{\cp}^{-1/4}\int_\cp
\ab{f}^{3/2}d\sigma\biggr)^{\delta/3},
\end{equation}
\end{prop}
\begin{proof}
 From Proposition \ref{dyadic-cap-estimate-for-T}, we have
\[\norma{Tf}_{L^6}\lesssim \Biggl(\suma{k}{}\norma{f_k}_{L^2}^{3-3\delta/2}
\biggl(\sup\limits_{\cp} \ab{\cp}^{-1/4}\int_\cp \ab{f_k}^{3/2}d\sigma\biggr)^{\delta}\Biggr)^{1/3}.\]
For each $k$, using that $\delta\leq 2/3$ ($\delta$ can be taken as small as desired by changing the corresponding implicit 
constants $C$ in the inequalities), we have
\begin{align*}
 \norma{f_k}_{L^2}^{3-3\delta/2}\biggl(\sup\limits_{\cp} \ab{\cp}^{-1/4}\int_\cp
\ab{f_k}^{3/2}d\sigma\biggr)^{\delta}&=\norma{f_k}_{L^2}^2 \norma{f_k}_{L^2}^{1-3\delta/2}\biggl(\sup\limits_{\cp}
\ab{\cp}^{-1/4}\int_\cp \ab{f_k}^{3/2}d\sigma\biggr)^{\delta}\\
&\leq \norma{f_k}_{L^2}^2 \norma{f}_{L^2}^{1-3\delta/2}\biggl(\sup\limits_{\cp} \ab{\cp}^{-1/4}\int_\cp
\ab{f}^{3/2}d\sigma\biggr)^{\delta}.
\end{align*}
Then, adding over $k$,
\[\norma{Tf}_{L^6}\lesssim \norma{f}_{L^2}^{1-\delta/2}\biggl(\sup\limits_{\cp} \ab{\cp}^{-1/4}\int_\cp
\ab{f}^{3/2}d\sigma\biggr)^{\delta/3}.\qedhere\]
\end{proof}

\section{Using the cap bound}\label{using-cap-bound}
We will prove the analog of \cite{CS}*{Lemma 2.6}.

\begin{lemma}
\label{cap-decomposition}
 For any $\delta>0$ there exist $C_\delta<\infty$ and $\eta_\delta>0$ with the following property. If $f\in L^2(\Gamma^2)$ 
 satisfies $\norma{Tf}_6\geq\delta \mathbf{C}\norma{f}_2$, then there exist a decomposition $f=g+h$ and a cap $\cp$ satisfying
\begin{align}
\label{cond1}
&0\leq \ab{g},\ab{h}\leq \ab{f},\\
\label{cond2}
& g, h \text{ have disjoint supports},\\
\label{cond3}
&\ab{g(x)}\leq C_\delta\norma{f}_2\ab{\cp}^{-1/2}\chi_\cp(x)\; \text{ for all } x,\\
\label{cond4}
&\norma{g}_2\geq \eta_\delta\norma{f}_2.
\end{align}
\end{lemma}

\begin{proof}
For convenience, normalize so that $\norma{f}_{L^2(\Gamma^2)}=1$. By Proposition \ref{cap-estimate-for-T} there exists a cap 
$\cp$ such that 
\[\int_\cp \ab{f}^{3/2}drd\te\geq \tfrac{1}{2}c(\delta)\ab{\cp}^{1/4}.\]

Let $R\geq 1$ and define $E=\{x\in\cp:\ab{f(x)}\leq R\}$. Set $g=f\chi_E$ and $h=f-f\chi_E$. Then $g$ and $h$ have disjoint
supports, 
$g+h=f$, $g$ is supported on $\cp$, and $\norma{g}_\infty\leq R$. Since $\ab{h(x)}\geq R$ for almost every $x\in\cp$ for which
$h(x)\neq 0$ we have
\[\int_\cp \ab{h}^{3/2}\leq R^{-1/2}\int_\cp h^2\leq R^{-1/2}\norma{f}_2^2=R^{-1/2}.\]

If we choose $R$ by setting $R^{-1/2}=\tfrac{1}{4}c(\delta)\ab{\cp}^{1/4}$, then
\[\int_\cp\ab{g}^{3/2}=\int_\cp\ab{f}^{3/2}-\int_\cp\ab{h}^{3/2}\geq \tfrac{1}{4}c(\delta)\ab{\cp}^{1/4}.\]

By H\"{o}lder's inequality, since $g$ is supported on $\cp$,
\[\norma{g}_2\geq \ab{\cp}^{-1/6}\biggl(\int\ab{g}^{3/2}\biggr)^{2/3}\geq c'(\delta)=c'(\delta)\norma{f}_2>0.\qedhere\]
\end{proof}

We note that the conditions $\ab{g(x)}\leq C_\delta\norma{f}_2\ab{\cp}^{-1/2}\chi_\cp(x)$ and 
$\norma{g}_2\geq \eta_\delta\norma{f}_2$ easily imply a lower bound on the $L^1$ norm of $g$.

\begin{lemma}
\label{large-l1}
 Let $g\in L^2(\Gamma^2)$ satisfying $\ab{g(x)}\leq a\ab{\cp}^{-1/2}\chi_\cp(x)$ and $\norma{g}_2\geq b$, for some $a,b>0$ and 
 $\cp\subset\Gamma^2$. Then there is a constant $C=C(a,b)>0$ such that 
\[\norma{g}_{L^1(\Gamma^2)}\geq C\ab{\cp}^{1/2}.\]
\end{lemma}

\begin{proof}
The hypotheses on $g$ imply that $\ab{a^{-1}\ab{\cp}^{1/2}g(x)}\leq \chi_\cp(x)\leq 1$ and thus\break
$\norma{a^{-1}\ab{\cp}^{1/2}g}_2^2\leq \norma{a^{-1}\ab{\cp}^{1/2}g}_1$. Therefore
\[\norma{g}_1\geq a^{-1}\ab{\cp}^{1/2}\norma{g}_2^2\geq a^{-1}b^2\ab{\cp}^{1/2}.\qedhere\]
\end{proof}

\section{Using the group of symmetries}\label{lorentz-invariance}
A Lorentz transformation, $L$, in $\R^3$ is an invertible linear map that preserves the bilinear form 
\[A(x,y)=x_1y_1+x_2y_2-x_3y_3,\]
$x=(x_1,x_2,x_3),y=(y_1,y_2,y_3)\in\R^3$, that is,
\[A(x,y)=A(Lx,Ly)\quad\text{ for all }x,y\in\R^3.\]
Examples of Lorentz transformations are $L^t,\,M^t$ and $R_\te$ given next. For $t\in(-1,1)$, we define the linear map 
$L^t:\R^3\to\R^3$ by 
\[L^t(x_1,x_2,x_3)=\biggl(\frac{x_1+tx_3}{\sqrt{1-t^2}},x_2,\frac{x_3+tx_1}{\sqrt{1-t^2}}\biggr).\]
$\{L^t\}_{t\in(-1,1)}$ is a one parameter subgroup of Lorentz transformations. Similarly,
\[M^t(x_1,x_2,x_3)=\biggl(x_1,\frac{x_2+tx_3}{\sqrt{1-t^2}},\frac{x_3+tx_2}{\sqrt{1-t^2}}\biggr)\]
is a Lorentz transformation.

One computes that $L^t$ and $M^t$ preserve the cone for all $t\in(-1,1)$, that is, $L^t(\Gamma^2)=M^t(\Gamma^2)=\Gamma^2$. 
For $\la>0$, we define the dilation $D_\la:\R^3\to\R^3$ by $D_\la(x)=\la x$ that clearly satisfies $D_\la(\Gamma^2)=\Gamma^2$ for
every $\la>0$. For $\te\in[0,2\pi]$, we denote by $R_\te$ the rotation in $\R^3$ by angle $\te$ about the $x_3$-axis
\[R_\te(x_1,x_2,x_3)=(x_1\cos\te-x_2\sin\te,x_1\sin\te+x_2\cos\te,x_3).\]
$R_\te$ preserves the cone for all $0\leq \te\leq 2\pi$.

Associated to $L^t,\,M^t,\,D_\la,\,R_\te$ are the operators $L^{t*},\,M^{t*},\,D_\la^*$ and $R_\te^*$ acting on a function 
$f\in L^2(\Gamma^2)$ by
\begin{equation}
 \label{group-action}
L^{t*}f=f\circ L^t,\; M^{t*}f=f\circ M^t,\; D_{\la}^*f=\la^{1/2}f\circ D_\la,\; R_{\te}^*f=f\circ R_\te,
\end{equation}
where ``$\circ$'' denotes composition. We also define $L_t=\sqrt{1-t^2}L^t=D_{\sqrt{1-t^2}} L^t$ and $L^*_t$ by
\begin{align*}
L_t^*f(x_1,x_2,x_3)&=(1-t^2)^{1/4}f\circ L_t(x_1,x_2,x_3)\\
&=(1-t^2)^{1/4}f(x_1+tx_3,\sqrt{1-t^2}\,x_2,x_3+tx_x).
\end{align*}

The measure $\sigma$ is invariant under the action of Lorentz transformations that preserve the cone, and in fact is the only 
one with that property, up to multiplication by constant; for this we refer to \cite{RS} where the case of the cone in $\R^4$ is
considered. In this paper we only need to know that for every $t\in (-1,1)$, $L^t$ and $M^t$ preserve the measure $\sigma$, and
this can be done directly using the change of variables formula and seeing that the Jacobians work out. We write it in the next
proposition and include the proof for completeness.
\begin{prop}
 For any $t\in(-1,1)$ the linear maps $L^t,\,M^t$ are invertible, preserve $\Gamma^2$, that is,
$L^t(\Gamma^2)=M^t(\Gamma^2)=\Gamma^2$, and preserve $\sigma$, that is, for any $f\in L^1(\Gamma^2)$
\[\int_{\Gamma^2} f\circ L^t d\sigma=\int_{\Gamma^2} f\circ M^t d\sigma=\int_{\Gamma^2} f d\sigma.\]
\end{prop}
\begin{proof}
 Letting $P(x_1,x_2,x_3)=(x_2,x_1,x_3)$ we see that $M^t=P\circ L^t\circ P$ and so it is enough to prove the statements for 
 $L^t$. The inverse of $L^t$ is $L^{-t}$. That $L^t(\Gamma^2)\subseteq \Gamma^2$ follows from the equality
\[\biggl(\frac{x_3+tx_1}{\sqrt{1-t^2}}\biggr)^2=\biggl(\frac{x_1+tx_3}{\sqrt{1-t^2}}\biggr)^2+x_2^2\]
and the inequality
\[\frac{x_3+tx_1}{\sqrt{1-t^2}}\geq 0\]
whenever $x_3^2=x_1^2+x_2^2$ and $x_3\geq 0$. Since the same is true for $L^{-t}$, it follows that $L^t(\Gamma^2)=\Gamma^2$. 
For the invariance of the measure, let $f\in L^1(\Gamma^2)$. We have
\begin{align*}
 \int f\circ L^t(x_1,x_2,x_3)d\sigma(x_1,x_2,x_3)=\int_{\R^2} f\Bigl(\frac{y_1+ty_3}{\sqrt{1-t^2}},y_2,
        \frac{y_3+ty_1}{\sqrt{1-t^2}}\Bigr)\frac{dy_1 dy_2}{\sqrt{y_1^2+y_2^2}}
\end{align*}
where $y_3=\sqrt{y_1^2+y_2^2}$. We use the change of variables $u=(y_1+ty_3)/\sqrt{1-t^2}=
(y_1+t\sqrt{y_1^2+y_2^2})/\sqrt{1-t^2},\, v=y_2$. We note that the Jacobian is
\[\frac{\partial(u,v)}{\partial(y_1,y_2)}=\frac{1}{\sqrt{1-t^2}}\Biggl(1+\frac{ty_1}{\sqrt{y_1^2+y_2^2}}\Biggr),\]
or equivalently
\[\frac{\partial(y_1,y_2)}{\partial(u,v)}=\sqrt{y_1^2+y_2^2}\frac{\sqrt{1-t^2}}{\sqrt{y_1^2+y_2^2}+ty_1}.\]
Now, since $L^t(y_1,y_2,y_3)$ lies in $\Gamma^2$, we also have
\[\sqrt{u^2+v^2}=\frac{y_3+ty_1}{\sqrt{1-t^2}}.\]
It follows that the Jacobian factor can be rewritten as
\[\frac{\partial(y_1,y_2)}{\partial(u,v)}=\frac{\sqrt{y_1^2+y_2^2}}{\sqrt{u^2+v^2}}.\]
Therefore, letting $w=\sqrt{u^2+v^2}$,
\[\int_{\R^2} f\biggl(\frac{y_1+ty_3}{\sqrt{1-t^2}},y_2,\frac{y_3+ty_1}{\sqrt{1-t^2}}\biggr)\frac{dy_1 dy_2}{\sqrt{y_1^2+y_2^2}}=
  \int_{\R^2} f(u,v,w)\frac{du dv}{\sqrt{u^2+v^2}}\]
or equivalently,
\[\int f\circ L^t d\sigma=\int fd\sigma.\qedhere\]
\end{proof}

The Lorentz invariance of the measure implies invariance of the $L^2$ norm, for $f\in L^2(\Gamma^2)$
\begin{equation}
 \label{invariance-norm}
 \norma{L^{t*}f}_{L^2(\sigma)}=\norma{M^{t*}f}_{L^2(\sigma)}=\norma{D_{\la}^*f}_{L^2(\sigma)}=\norma{R_\te^*f}_{L^2(\sigma)}
  =\norma{L_t^*f}_{L^2(\sigma)}=\norma{f}_{L^2(\sigma)}.
\end{equation}

Using the Lorentz invariance of $\sigma$, it is direct to check that for all $p\in[1,\infty]$ the $L^p$ norm of $Tf$ does not 
change under Lorentz transformations in the sense that 
\begin{equation}
 \label{lorentz-invariance-T}
 \norma{T(f\circ L)}_{L^p(\R^3)}=\norma{Tf}_{L^p(\R^3)}.
\end{equation}
Indeed, writing
\[Tf(x,t)=\int e^{ix\cdot y}e^{ity'}f(y,y') d\sigma(y,y')=\int e^{iA((x,-t),(y,y'))}f(y,y') d\sigma(y,y'),\]
we obtain
\begin{align*}
 T(f\circ L)(x,t)&=\int e^{iA((x,-t),(y,y'))}f\circ L(y,y') d\sigma(y,y')\\
 &=\int e^{iA(L(x,-t),L(y,y'))}f\circ L(y,y') d\sigma(y,y')\\
 &=\int e^{iA(L(x,-t),(y,y'))}f(y,y') d\sigma(y,y').
\end{align*}
Since for a Lorentz transformation $L$, $\ab{\det L}=1$, \eqref{lorentz-invariance-T} follows by change of variables in 
the case $p\in[1,\infty)$. When $p=\infty$, \eqref{lorentz-invariance-T} follows since $L$ is invertible.

We can use the group of symmetries to transform a cap into a set comparable to a cap of large measure, that is, we have the
following lemma: 
\begin{lemma}
\label{wider-cap}
 Let $\cp\subset[\frac{1}{2},1]\times [0,2\pi]$ be a cap in $\Gamma^2$. Then there exist $t\in[0,1)$ and $\te\in[0,2\pi]$ such
that 
 $L_t^{-1}R_\te^{-1}(\cp)$ satisfies
 \begin{equation}
  \label{big-measure}
  \sigma(L_t^{-1}R_\te^{-1}(\cp))\geq \tfrac{1}{2}\quad\text{ and }\quad 
L_t^{-1}R_\te^{-1}(\cp)\subseteq[\tfrac{1}{4},1]\times[0,2\pi].
 \end{equation}
\end{lemma}
\begin{proof}
Let $\te\in[0,2\pi]$ be such that $R_\te^{-1}\cp=[\frac{1}{2},1]\times [-\eps,\eps]$, for some $\eps\in[0,\pi]$. The measure of
$\cp$ 
is $\sigma(\cp)=\ab{\cp}=\eps$, and so we can assume $\eps< \frac{1}{2}$, otherwise we are done by taking $t=0$. 

The inverse of $L_t$ is $L_t^{-1}=(1-t^2)^{-1/2}L^{-t}$ and the measure of $L_t^{-1}R_\te^{-1}(\cp)$ is 
\begin{align*}
 \sigma(L_t^{-1}R_\te^{-1}(\cp))&=\sigma((1-t^2)^{-1/2}L^{-t}R_\te^{-1}(\cp))\\
 &=\sigma((L^t)^{-1}R_\te^{-1}((1-t^2)^{-1/2}\cp))\\
 &=\sigma((1-t^2)^{-1/2}\cp)=(1-t^2)^{-1/2}\sigma(\cp),
\end{align*}
where we used the invariance of $\sigma$ under Lorentz transformations and that $\sigma(\la\cp)=\la\sigma(\cp)$ for any 
$\la\geq 0$.

Let $t$ be such that $\sigma(L_t^{-1}R_\te^{-1}(\cp))=1$, that is, $t= (1-\ab{\cp}^2)^{1/2}=(1-\eps^2)^{1/2}$. Now we write
$R_\te^{-1}\cp=\{(r\cos\varphi,r\sin\varphi,r):\frac{1}{2}\leq r\leq 1, -\eps\leq\varphi\leq \eps\}$ , so that
\[ L_t^{-1}R_\te^{-1}(\cp)=\biggl\{r(1-t^2)^{-\frac{1}{2}}\biggl(\frac{\cos\vphi-t}{(1-t^2)^{1/2}},\sin\vphi,\frac{1-t\cos\vphi}{
(1-t^2)^ { 1/2 }}\biggr):\frac{1}{2}\leq r\leq 1,-\eps\leq \vphi\leq\eps\biggr\}.\]
Note that for $\frac{1}{2}\leq r\leq 1$ and $-\eps\leq \vphi\leq \eps$ we have
\[r(1-t^2)^{-1/2}\frac{\ab{\cos\vphi-t}}{(1-t^2)^{1/2}}\leq \frac{1-t}{1-t^2}=\frac{1}{1+t}\leq 1\]
because $\cos\varphi\geq\cos\eps\geq t$. Similarly
\[r(1-t^2)^{-1/2}\ab{\sin\vphi}\leq \frac{\sin\eps}{\eps}\leq 1\]
and
\[\frac{1}{4}\leq \frac{1}{2(1+t)}=\frac{1-t}{2(1-t^2)}\leq r\frac{1-t\cos\vphi}{1-t^2}\leq 1.\]

Then $t=(1-\eps^2)^{1/2}$ gives the desired conclusion.
\end{proof}

\begin{cor}
\label{appplying-lorentz-group}
 Let $\{f_n\}_{n\in\N}$ be a sequence of nonnegative functions in $L^2(\Gamma^2)$ with $\norma{f_n}_{L^2(\Gamma^2)}=1$ 
 and such that there exists a cap $\cp_n\subset [1/2,1]\times [0,2\pi]$ with the property
 \begin{equation}
  \label{lower-bound-on-cap}
  \int_{\cp_n}f_n d\sigma\geq c\ab{\cp_n}^{1/2},
 \end{equation}
 where $c>0$ is independent of $n$. Then there exist sequences $\{t_n\}_{n\in\N}\subset[0,1)$ and
$\{\te_n\}_{n\in\N}\subset[0,2\pi]$ such that $\{L_{t_n}^*R_{\te_n}^*f_n\}_{n\in\N}$ satisfies that every weak limit in
$L^2(\Gamma^2)$ is nonzero.
\end{cor}
\begin{proof}
 $L_t^*$ and $R_\te^*$ preserve the $L^2(\Gamma^2)$ norm thus $\norma{L_t^*R_\te^*f_n}_{L^2(\Gamma^2)}=1$ for any $t\in[0,1)$ and
$\te\in[0,2\pi]$. It follows that for any sequences $\{t_n\}_{n\in\N}\subset[0,1)$ and $\{\te_n\}_{n\in\N}\subset[0,2\pi]$, the
set of $L^2$-weak limits of $\{L_{t_n}^*R_{\te_n}^*f_n\}_{n\in\N}$ is nonempty. 
 
 Under the action of $L_t^*R_\te^*$, the integral of a function $f$ changes according to
 \begin{equation}
  \label{l1-norm-under-symmetry}
  \int L_t^*R_\te^*f d\sigma=(1-t^2)^{-1/4}\int R_\te^*f d\sigma=(1-t^2)^{-1/4}\int f d\sigma.
 \end{equation}
 
 By Lemma \ref{wider-cap}, for each $n$, there exist $t_n\in[0,1)$ and $\te_n\in[0,2\pi]$ such that 
 \[\sigma(L_{t_n}^{-1}R_{\te_n}^{-1}(\cp_n))\geq \tfrac{1}{2}\quad\text{ and
}\quad L_{t_n}^{-1}R_{\te_n}^{-1}(\cp_n)\subseteq[\tfrac{1}{4},1]\times[0,2\pi].\]
 
 Suppose that for a subsequence of $\{f_n\}_{n\in\N}$ (that we also call $\{f_n\}_{n\in\N}$)
$L_{t_n}^*R_{\te_n}^*f_n\rightharpoonup f$, as $n\to\infty$, for some
$f\in L^2(\Gamma^2)$.
Using \eqref{lower-bound-on-cap} and \eqref{l1-norm-under-symmetry}, we have
 \begin{multline}
  \label{lower-bound-pull-back}
  \int_{[1/4,1]\times[0,2\pi]} L_{t_n}^*R_{\te_n}^*f_n d\sigma\geq (1-t_n^2)^{-1/4}\int_{\cp_n} f_n d\sigma\\
\geq c(1-t_n^2)^{-1/4}\ab{\cp_n}^{1/2}=c(\sigma(L_{t_n}^{-1}R_{\te_n}^{-1}(\cp_n)))^{1/2}\geq \frac{c}{\sqrt{2}}.
 \end{multline}

  From \eqref{lower-bound-pull-back} and the weak convergence, it follows that
 \[\int_{[1/4,2]\times[0,2\pi]}f d\sigma\geq \frac{c}{\sqrt{2}}>0\]
 and so $f\neq 0$.
\end{proof}

\section{The proof of the precompactness}\label{fvv-argument}

In this section, we prove that up to symmetries of the cone, an extremizing sequence is precompact. 

We will give two proofs. We will start with a proof based on \cite{FVV,FVV2}; the other proof is based on 
\cite{CS} with a modification coming from \cite{FVV} and is sketched in a remark at the end of this section.

Recall that $\textbf{C}$, given in \eqref{best-constant}, denotes
the best constant in inequality \eqref{restriction-cone-with-c}, in other words, $\mathbf{C}=\norma{T}$, the norm of the
operator $T$ as a map from $L^2(\Gamma^2)$ to $L^6(\R^3)$.

We start by stating \cite{FVV}*{Proposition 1.1} for the cone.

\begin{prop}\cite{FVV}.
\label{cone-hilbert-lemma-l2}
 Let $T:L^2(\Gamma^2,\sigma)\to L^6(\R^3)$ be the Fourier extension operator defined in \eqref{fourier-extension-operator}. 
 Let $\{f_n\}_{n\in\N}\subset L^2(\Gamma^2)$ such that:
\begin{itemize}
 \item[(i)] $\lim\limits_{n\to\infty}\norma{f_n}_2= 1$;
\item[(ii)] $\lim\limits_{n\to\infty}\norma{Tf_n}_{L^6(\R^3)}=\mathbf{C}$;
\item[(iii)] $f_n\rightharpoonup f\neq 0$;
\item[(iv)] $Tf_n\to Tf\text{ a.e. in }\R^3$.
\end{itemize}
Then $f_n\to f$ in $L^2(\Gamma^2)$, in particular $\norma{f}_2=1$ and $\norma{Tf}_{L^6(\R^3)}=\mathbf{C}$.
\end{prop}

We have changed slightly condition (i) in \cite{FVV}*{Proposition 1.1} from $\norma{f_n}_2= 1$ to
$\lim_{n\to\infty}\norma{f_n}_2= 1$, but the proposition as stated here is easily shown to be equivalent to the one in
\cite{FVV} by considering $f_n/\norma{f_n}_2$. Note that an extremizing sequence $\{f_n\}_{n\in\N}$ as in Definition
\ref{def-ext-sequence} satisfies (i) and (ii) in the previous proposition.

We now restate the precompactness theorem, Theorem \ref{main-precompactness}, in a
more precise way,

\begin{teo}
  \label{precompactness}
Let $\{f_n\}_{n\in\N}$ be an extremizing sequence for \eqref{restriction-cone-with-c} of nonnegative functions in $L^2(\Gamma^2)$.
Then there exist sequences $\{t_n\}_{n\in \N}\subset (-1,1)$, $\{\te_n\}_{n\in\N}\subset[0,2\pi]$ and
$\{\la_n\}_{n\in\N}\subset(0,\infty)$ such that $\{L_{t_n}^*R_{\te_n}^*D_{\la_n}^*f_n\}_{n\in\N}$ is precompact, that is, any
subsequence has a convergent sub-subsequence in $L^2(\Gamma^2)$.
\end{teo}
\begin{proof}
Since $\{f_n\}_{n\in\N}$  is an extremizing sequence, for all $n$ large enough $\norma{Tf_n}_6\geq
\mathbf{C}\norma{f_n}_2/2$. By Lemma \ref{cap-decomposition} with $\delta=\frac{1}{2}$ there exists $C<\infty$ and $\eta>0$, a
decomposition $f_n=g_n+h_n$ and a cap $\cp_n$ satisfying \eqref{cond1}-\eqref{cond4}. Using that
$\norma{f_n}_{L^2}\to1$, as $n\to\infty$, and Lemma \ref{large-l1} for $g_n$ gives
\[\norma{g_n}_{L^1(\Gamma^2)}\geq C\ab{\cp_n}^{1/2},\]
where $C$ is independent of $n$.

Now there exists $\{\la_n\}_{n\in\N}\subset (0,\infty)$ such that $\la_n^{-1}\cp_n\subset [\frac{1}{2},1]\times [0,2\pi]$ and 
$\la_n^{-1}\cp_n$ is a cap as in Definition \ref{definition-of-cap}. By dilation invariance $\{D_{\la_n}^*f_n\}_{n\in\N}$ is also
an extremizing sequence, with $\norma{D_{\la_n}^*f_n}_2=\norma{f_n}_2$.  The decomposition for $f_n$ gives a decomposition for
$D_{\la_n}^*f_n$, $D_{\la_n}^*f_n=D_{\la_n}^*g_n+D_{\la_n}^*h_n$, and
\[\int_{\la_n^{-1}\cp_n} D_{\la_n}^*f_n d\sigma\geq \norma{D_{\la_n}^*g_n}_{L^1(\Gamma^2)}\geq C\ab{\la_n^{-1} \cp_n}^{1/2}.\]
We now apply Corollary \ref{appplying-lorentz-group} to $\{D_{\la_n}^*f_n\}_{n\in\N}$ to obtain sequences
$\{t_n\}_{n\in\N}\subset[0,1)$ 
and $\{\te_n\}_{n\in\N}\subset[0,2\pi]$ such that every weak limit of $\{L_{t_n}^*R_{\te_n}^*D_{\la_n}^*f_n\}_{n\in\N}$ in
$L^2(\Gamma^2)$ is nonzero.

In view of Proposition \ref{cone-hilbert-lemma-l2}, the theorem is proved if we show that, after passing to a subsequence, if
$L_{t_n}^*R_{\te_n}^*D_{\la_n}^*f_n\rightharpoonup f$, as $n\to\infty$, then $TL_{t_n}^*R_{\te_n}^*D_{\la_n}^*f_n\to Tf$ a.e. in
$\R^3$. We will do this by using the following proposition:
\end{proof}

\begin{prop}
\label{from-fvv2}
 Let $\{u_n\}_{n\in\N}$ be a uniformly bounded sequence in $L^2(\Gamma^2)$, that is, $\sup_n\norma{u_n}_2=:c<\infty$. 
 Suppose there exists $u\in L^2(\Gamma^2)$ such that $u_n\rightharpoonup u$ as $n\to\infty$. Then, there exists a subsequence
$\{u_{n_k}\}_{k\in\N}$ such that $Tu{_{n_k}}\to Tu$ almost everywhere in $\R^3$.
\end{prop}
\begin{proof}
 The proof follows \cite{FVV2}*{Proof of Theorem 1.1}. We record it here for the completeness. We start by defining $v_n(y)$ by
its Fourier transform
\[\hat v_n(y)=u_n(y)\ab{y}^{-1},\]
and $\hat v(y)=u(y)\ab{y}^{-1}$.

Since $\norma{u_n}_{L^2(\Gamma^2)}^2=\int_{\R^2}\ab{u_n(y)}^2\ab{y}^{-1}dy\leq c^2$, we see that $\norma{v_n}_{\dot
H^{1/2}(\R^2)}^2=\break\int_{\R^2}\ab{\hat v_n(y)}^2\ab{y}dy\leq c^2$. The operator $T$ applied to $u_n$ equals
$(2\pi)^2e^{it\sqrt{-\Delta}}v_n$. Fix
$t\in \R$, by the continuity of $e^{it\sqrt{-\Delta}}$ in $\dot H^{1/2}(\R^2)$, we have 
\[e^{it\sqrt{-\Delta}}v_n\rightharpoonup e^{it\sqrt{-\Delta}}v\]
weakly in $\dot H^{1/2}(\R^2)$, as $n\to\infty$. Then, by the Rellich Theorem \cite{AN}*{Theorem 1.5}, for any $R>0$
\[e^{it\sqrt{-\Delta}}v_n\to e^{it\sqrt{-\Delta}}v\]
strongly in $L^2(B(0,R))$, as $n\to\infty$. Denote by
\[F_n(t):=\int_{\ab{x}<R}\abs{e^{it\sqrt{-\Delta}}(v_n-v)}^2dx=\norma{e^{it\sqrt{-\Delta}}(v_n-v)}_{L^2(B(0,R))}^2.\]
By H\"older's inequality and Sobolev embedding, $\dot H^{1/2}(\R^2)\subset L^4(\R^2)$, we obtain
\[F_n(t)\leq CR\norma{e^{it\sqrt{-\Delta}}(v_n-v)}_{\dot H^{1/2}(\R^2)}^2\leq 2CR,\]
consequently, by the Fubini and dominated convergence Theorems we have that
\[\int_{-R}^R F_n(t)dt=\int_{-R}^{R}\int_{\ab{x}<R}\abs{e^{it\sqrt{-\Delta}}(v_n-v)}^2dxdt\to 0\]
as $n\to \infty$. This implies that, up to a subsequence,
\[e^{it\sqrt{-\Delta}}(v_n-v)\to 0\quad\text{a.e. in }B(0,R)\times (-R,R).\]
Repeating the argument on a discrete sequence of radii $R_n$ such that $R_n\to\infty$, as $n\to\infty$ we conclude, by a 
diagonal argument, that there exists a subsequence $v_{n_k}$ of $v_n$ such that
\[e^{it\sqrt{-\Delta}}(v_{n_k}-v)(x)\to 0\quad\text{a.e. for }(x,t)\in\R^2\times \R,\]
or equivalently, in terms of the sequence $\{u_n\}_{n\in\N}$,
\[Tu_{n_k}-Tu\to 0\quad\text{a.e. in }\R^3.\qedhere\]
\end{proof}
This concludes the proof of Theorem \ref{precompactness}. 

\begin{remark}
To end this section we want to give an alternative proof of Theorem \ref{main-precompactness} that uses the
Christ-Shao concentration-compactness argument \cite{CS} to gain control over extremizing sequences. We will give a brief sketch
of the argument, indicating the results from \cite{CS} that we need but we will not go into the details of their proofs as they
can be adapted to the case of the cone without much effort (for the details, we refer the interested reader to the version of this
work available on the arXiv, \url{http://arxiv.org}). 

We have already obtained the analogs of the results of \cite{CS}*{Section 6}, namely Proposition \ref{cap-estimate-for-T} and
Lemma \ref{cap-decomposition}

From \cite{CS}*{Section 7}, we need analogs of Lemmas 7.1 and 7.5. To obtain a result as \cite{CS}*{Lemma 7.5} we define a
metric space associated to the set of caps.

Let $\mathcal M$ be the set of all caps of $\Gamma^2$ as in Definition \ref{definition-of-cap}, modulo the equivalence relation
$\cp\sim\cp'$ if there exists $k\in \Z$ such that $\cp,\cp'\subseteq[2^{k-1},2^k]\times [0,2\pi]$. The equivalent class of a cap
$\cp=[2^{k-1},2^k]\times J$ is $[\cp]=\{[2^{k-1},2^k]\times I: I\subseteq[0,2\pi] \text{ and }I \text{ is an interval}\}$. 

We make $\mathcal M$ into a metric space by defining the distance between the equivalent classes of $\cp,\,\cp'$ by
 \[\varrho([\cp],[\cp'])=\ab{k-k'},\]
where $\cp=[2^{k-1},2^k]\times J$ and $\cp'=[2^{k'-1},2^{k'}]\times J'$. 

The equivalent of \cite{CS}*{Lemma 7.5} is the bilinear estimate in Lemma \ref{bilinear-estimate}, as it implies that for caps
$\cp,\,\cp'\subseteq\Gamma^2$ the following estimate holds with a universal constant $C<\infty$:
\[\norma{T\chi_\cp\cdot T\chi_{\cp'}}_{L^3(\R^3)}\leq C2^{-\varrho([\cp],[\cp'])/6}\ab{\cp}^{1/2}\ab{\cp'}^{1/2}.\]

We now move to \cite{CS}*{Section 8}, the decomposition algorithm. Note that the decomposition algorithm
does not depend on the specific manifold we are dealing with, it just requires Lemma \ref{cap-decomposition}. Useful properties
can be obtained from the decomposition algorithm for functions which are close to be an extremizer for
\eqref{restriction-cone-with-c} and they appear in \cite{CS}*{Sections 8 and 9}. 

The main proposition we obtain from the Christ-Shao argument is the equivalent of \cite{CS}*{Proposition 2.7} whose
proof is contained in \cite{CS}*{Section 10}.

\begin{prop}
 \label{decay-control-delta-extremal}
There exists a function $\Theta:[1,\infty)\to(0,\infty)$ satisfying $\Theta(R)\to 0$ as $R\to\infty$ with the following property. 
For any $\eps>0$, there exists $\delta>0$ such that any nonnegative function $f\in L^2(\Gamma^2)$ satisfying 
\[\norma{f}_{L^2(\Gamma^2)}=1\;\text{ and }\;\norma{Tf}_{L^6(\R^3)}\geq (1-\delta)\mathbf{C}\norma{f}_{L^2(\Gamma^2)}\]
may be decomposed as $f=F+G$ where $F,G$ are
nonnegative with disjoint supports, $\norma{G}_{L^2(\Gamma^2)}<\eps$,
and there exists $k\in\Z$ such that 
\begin{equation}
\label{decay-at-infinity}
 \int_{\ab{y}<2^{k}R^{-1}}\ab{F(y)}^2d\sigma(y)+\int_{\ab{y}>2^{k}R}\ab{F(y)}^2d\sigma(y)\leq \Theta(R)\quad\text{for all } R\geq
1.
\end{equation}
\end{prop}

The last ingredient we need is the proposition below whose proof can be obtained by a slight modification of 
\cite{FVV}*{Proof of Proposition 1.1}.

\begin{prop}
\label{modified-cone-hilbert-lemma-l2}
 Let $T:L^2(\Gamma^2,\sigma)\to L^6(\R^3)$ be the Fourier extension operator defined in \eqref{fourier-extension-operator}. Let
$\{f_n\}_{n\in\N}\subset L^2(\Gamma^2)$
such that:
\begin{itemize}
 \item[(i)] $\lim\limits_{n\to\infty}\norma{f_n}_2=1$;
\item[(ii)] $\lim\limits_{n\to\infty}\norma{Tf_n}_{L^6(\R^3)}=\mathbf{C}$;
\item[(iii)] $f_n\rightharpoonup f\neq 0$;
\item[(iv)] $\sup\limits_{n\in \N}\norma{f_n}_{L^2(\{\ab{y}\geq R\})}\leq \Theta(R)$, where $\Theta(R)\to 0$ as $R\to\infty$.
\end{itemize}
Then $f_n\to f$ in $L^2(\Gamma^2)$, in particular $\norma{f}_2=1$ and $\norma{Tf}_{L^6(\R^3)}=\mathbf{C}$.
\end{prop}

Proposition \ref{modified-cone-hilbert-lemma-l2} is just as Proposition \ref{cone-hilbert-lemma-l2} except for condition (iv).

\begin{proof}[Proof (Sketch of the alternative proof of Theorem \ref{precompactness})]
 Let $\{f_n\}_{n\in \N}$ be an extremizing sequence of nonnegative functions. We start as in the proof of Theorem
\ref{precompactness} by using Lemma \ref{cap-decomposition}, Lemma \ref{wider-cap} and Corollary \ref{appplying-lorentz-group} to
obtain sequences $\{\la_n\}_{n\in\N}$, $\{t_n\}_{n\in\N}$ and $\{\te_n\}_{n\in\N}$ such that $\tilde
f_n:=L_{t_n}^*R_{\te_n}^*D_{\la_n}^*f_n$ satisfies
\begin{equation}
\label{fn-tilde}
\norma{\tilde f_n\chi_{\{\frac{1}{4}\leq \ab{y}\leq 1\}}}_2> c\;\text{ and }\;\int \tilde f_n\chi_{\{\frac{1}{4}\leq
\ab{y}\leq 1\}}d\sigma(y)>c,
\end{equation}
for all $n$ with a constant $c>0$ independent of $n$.

We now apply Proposition \ref{decay-control-delta-extremal} to
$\{\tilde f_n/\norma{\tilde f_n}_2\}_{n\in\N}$ with $\eps_n=1/n$, $n\geq 1$, to  
obtain a subsequence of $\{\tilde f_n\}_{n\in\N}$ (that we also call $\{\tilde f_n\}_{n\in\N}$ ), that satisfies the following.
Each $\tilde f_n$ can be
decomposed as $\tilde f_n=F_n+G_n$, with $F_n,G_n$ nonnegative with disjoint supports, $\norma{G_n}_2<\frac{1}{n}$ and $F_n$
satisfies \eqref{decay-at-infinity} for certain $k=k_n\in\Z$. 

We see that as $\{\tilde f_n\}_{n\in\N}$ is an extremizing sequence of
nonnegative functions for \eqref{restriction-cone-with-c}, so is  $\{F_n\}_{n\in\N}$, and we
claim it satisfies the hypotheses of Proposition \ref{modified-cone-hilbert-lemma-l2}, after passing to a subsequence if
necessary.

From \eqref{fn-tilde}, it follows that for all $n$ large enough $F_n$ satisfies
\begin{equation}
\label{properties-of-F}
\norma{F_n\chi_{\{\frac{1}{4}\leq\ab{y}\leq 1\}}}_2> \frac{c}{2}\,\text{ and }\,\int F_n\chi_{\{\frac{1}{4}\leq
\ab{y}\leq 1\}} d\sigma(y)>\frac{c}{2}.
\end{equation}
The first inequality in \eqref{properties-of-F} together with the $L^2$-decay estimate \eqref{decay-at-infinity} imply that
$\{k_n\}_{n\in\N}$ is a bounded sequence. After passing to a subsequence, $F_n\rightharpoonup F$ for some $F\in L^2(\Gamma^2)$ and
$F\neq 0$ since the $F_n$'s satisfy the
second inequality in \eqref{properties-of-F}. Therefore $\{F_n\}_{n\in\N}$ satisfies all the hypotheses of Proposition
\ref{modified-cone-hilbert-lemma-l2}
and thus $F_n\to F$ in $L^2(\Gamma^2)$. Therefore $\tilde f_n\to F$ in $L^2(\Gamma^2)$ which shows that $\{f_n\}_{n\in\N}$ is
precompact up to symmetries of the cone as needed.
\end{proof}
\end{remark}

\section{On convergence of extremizing sequences}

In this section, we prove Theorem \ref{better-precompactness}. We start with a general discussion.

Let $(X,\mathcal B,\mu)$ be a measure space and let $G$ be a group acting on $L^2(X)$, with an action that preserves the
$L^2$
norm, that is $\norma{g^*f}_{L^2(X)}=\norma{f}_{L^2(X)}$ for all $g\in G$ and $f\in L^2(X)$. For an element $f\in L^2(X)$, we
consider its orbit under $G$,\break $G(f):=\{g^*f:g\in G\}$.
 \begin{prop}
\label{precompactness-to-convergence}
  Let $f\in L^2(X)$ and $\{f_n\}_{n\in\N}$ a sequence in $L^2(X)$ with the property that every subsequence has an $L^2$-convergent
subsequence whose limit lies on $G(f)$. Then there exists a sequence $\{g_n\}_{n\in\N}\subset G$ such that $g_n^*f_n\to f$  in
$L^2(X)$, as $n\to \infty$.
 \end{prop}
\begin{proof}
For each $n$, let $g_n\in G$ be such that
\[\norma{g_n^*f_n-f}_{L^2(X)}\leq \inf\limits_{g\in G}\norma{g^*f_n-f}_{L^2(X)}+\frac{1}{n}.\]

We show that $\{g_n^*f_n\}_{n\in\N}$ converges to $f$ by showing that every subsequence has a further subsequence that converges
to $f$.
Take a subsequence of $\{g_{n}^*f_{n}\}_{n\in\N}$ (that we also call $\{g_{n}^*f_{n}\}_{n\in\N}$). By hypothesis,
$\{f_{n}\}_{n\in\N}$ has a convergent subsequence (that
we also call $\{f_{n}\}_{n\in\N}$) to an element in $G(f)$. That is $f_n\to g^* f$, as $n\to\infty$, for some $g\in G$. By the
definition of
$g_n$ and the
invariance of the norm under the action of $G$, we obtain
\[\norma{g_n^*f_n-f}_{L^2(X)}\leq \norma{(g^{-1})^{*}f_n-f}_{L^2(X)}+\frac{1}{n}=\norma{f_n-g^*f}_{L^2(X)}+\frac{1}{n}\to 0\]
as $n\to\infty$.
\end{proof}

From Theorem \ref{extremizers-cone} the extremizers for \eqref{restriction-cone-with-c} are all of the form
\begin{equation}
 \label{general-extremizer}
 g(x_1,x_2,x_3)=e^{-ax_3-bx_2-cx_1+d},
\end{equation}
where $a,b,c,d\in\mathbb C$ and $\ab{(\Re b,\Re c)}<\Re a$, and here $x_3=\sqrt{x_1^2+x_2^2}$. As indicated in \cite{Fo}, any 
extremizer can be obtained from $g_0(x_1,x_2,x_3)=e^{-x_3}$ by applying Lorentz transformations and dilations.

We define $G$ as the group generated by $L^t,M^s$ and $D_r$, $s,t\in(-1,1)$, $r>0$ under composition. The action of $G$ is 
given by the action of the generators as in \eqref{group-action} : $L^{t*}f=f\circ L^t,\, M^{s*}f=f\circ M^{s}$ and
$D_{r}^*f=r^{1/2}f\circ D_r$. That $G$ preserves the $L^2(\Gamma^2)$ norm follows from the Lorentz invariance of $\sigma$.
\begin{lemma}
 The set of real, $L^2$-normalized extremizers for inequality \eqref{restriction-cone-with-c} equals the orbit of 
 $g_0(y)=\pi^{-1/2}e^{-\ab{y}}$, $y\in\R^2$, under the group $G$.
\end{lemma}
\begin{proof}
A computation shows
\begin{multline*}
L^t\circ M^s(x_1,x_2,x_3)=\\
\biggl(\frac{x_1+t(x_3+sx_2)/(1-s^2)^{1/2}}{(1-t^2)^{1/2}},\frac{x_2+sx_3}{(1-s^2)^{1/2}},
  \frac{((x_3+sx_2)/(1-s^2)^{1/2})+tx_1}{(1-t^2)^{1/2}}\biggr).
  \end{multline*}
Then
\begin{multline*}
g_0\circ L^t\circ M^s\circ D_r=\\
r^{1/2}\pi^{-1/2}\exp\Bigl(-\frac{rx_3}{(1-s^2)^{1/2}(1-t^2)^{1/2}}-\frac{srx_2}{(1-s^2)^{1/2} (1-t^2)^{1/2} } -\frac {
trx_1}{(1-t^2)^{1/2}}\Bigr).
\end{multline*}
Given $a>0$ and $b,c\in \R$ with $\ab{(b,c)}<a$, we want to solve the equations
\begin{align*}
 \frac{r}{(1-s^2)^{1/2}(1-t^2)^{1/2}}&=a,\\
\frac{sr}{(1-s^2)^{1/2}(1-t^2)^{1/2}}&=b,\\
\frac{tr}{(1-t^2)^{1/2}}&=c.
\end{align*}
Since $a\neq 0$ and $\ab{b}<a$, we have $b/a=s\in(-1,1)$. Also $c/a=t(1-s^2)^{1/2}$, so $t=c/(a(1-s^2)^{1/2})=
c/(a^2-b^2)^{1/2}$ and we see that $\ab{t}<1$. Finally $r=a(1-s^2)^{1/2}(1-t^2)^{1/2}=(a^2-b^2-c^2)^{1/2}$. The $L^2$-norm
is preserved by the action of $G$, thus a normalized, real extremizer $g(y)=e^{-a\ab{y}-by_2-cy_1+d}$ can be obtained from $g_0$
by
composing with $L^t\circ M^s\circ D_r$ for the computed values of $t,s$ and $r$.
\end{proof}
\begin{proof}[Proof of Theorem \ref{better-precompactness}]
From the previous discussion we have that the group $G$ gives all real extremizers as the orbit of $g_0$. Proposition
\ref{precompactness-to-convergence}, Theorems \ref{main-precompactness} and \ref{extremizers-cone} give a proof of Theorem
\ref{better-precompactness}.
\end{proof}
{\bf Acknowledgments.}
The author would like to thank his dissertation advisor, Michael Christ, for many helpful comments and suggestions; and
the anonymous referee for the time spent on the reading of our manuscript and for the comments that helped improve this
work.


\begin{thebibliography}{10}
\bib{Ca}{article}{
   author={Carneiro, Emanuel},
   title={A sharp inequality for the Strichartz norm},
   journal={Int. Math. Res. Not.},
   date={2009},
   number={16},
   pages={3127--3145},
   issn={1073-7928},
}
\bib{CS}{article}{
   author={Christ, Michael},
   author={Shao, Shuanglin},
   title={Existence of extremizers for a Fourier restriction inequality},
   journal={Preprint},
   date={2010},
   note={To appear in Anal. PDE}
}
\bib{AN}{article}{
   author={Cotsiolis, Athanase},
   author={Tavoularis, Nikolaos K.},
   title={Best constants for Sobolev inequalities for higher order
   fractional derivatives},
   journal={J. Math. Anal. Appl.},
   volume={295},
   date={2004},
   number={1},
   pages={225--236},
   issn={0022-247X},
}
\bib{FVV}{article}{
   author={Fanelli, Luca},
   author={Vega, Luis},
   author={Visciglia, Nicola},
   title={On the existence of maximizers for a family of restriction theorems},
   journal={Bull. London Math. Soc.},
   volume={43},
   number={4},
   date={2011},
   pages={811-817},
}
\bib{FVV2}{article}{
   author={Fanelli, Luca},
   author={Vega, Luis},
   author={Visciglia, Nicola},
   title={Existence of maximizers for Sobolev-Strichartz inequalities},
   journal={Adv. Math.},
   volume={229},
   number={3},
   date={2012},
   pages={1912-1923},
   issn={0001-8708},
}
\bib{Fo}{article}{
   author={Foschi, Damiano},
   title={Maximizers for the Strichartz inequality},
   journal={J. Eur. Math. Soc.},
   volume={9},
   date={2007},
   number={4},
   pages={739--774},
   issn={1435-9855},
}
\bib{MVV}{article}{
   author={Moyua, A.},
   author={Vargas, A.},
   author={Vega, L.},
   title={Restriction theorems and maximal operators related to oscillatory
   integrals in $\R^3$},
   journal={Duke Math. J.},
   volume={96},
   date={1999},
   number={3},
   pages={547--574},
}
\bib{ONeil}{article}{
   author={O'Neil, Richard},
   title={Convolution operators and $L(p,q)$ spaces},
   journal={Duke Math. J.},
   volume={30},
   date={1963},
   pages={129--142},
   issn={0012-7094},
}
\bib{Ramos}{article}{
   author={Ramos, Javier},
   title={A refinement of the Strichartz inequality for the wave equation with applications},
   journal={Adv. Math.},
   volume={230},
   number={2},
   pages={649--698},
   date={2012},
}
\bib{RS}{book}{
   author={Reed, Michael},
   author={Simon, Barry},
   title={Methods of modern mathematical physics. II. Fourier analysis,
   self-adjointness},
   publisher={Academic Press [Harcourt Brace Jovanovich Publishers]},
   place={New York},
   date={1975},
   pages={xv+361},
}
\bib{So}{book}{
   author={Sogge, Christopher D.},
   title={Lectures on nonlinear wave equations},
   series={Monographs in Analysis, II},
   publisher={International Press},
   place={Boston, MA},
   date={1995},
   pages={vi+159},
   isbn={1-57146-032-2},
}
\bib{SW}{book}{
   author={Stein, Elias M.},
   author={Weiss, Guido},
   title={Introduction to Fourier analysis on Euclidean spaces},
   note={Princeton Mathematical Series, No. 32},
   publisher={Princeton University Press},
   place={Princeton, N.J.},
   date={1971},
   pages={x+297},
}
\bib{Str}{article}{
   author={Strichartz, Robert S.},
   title={Restrictions of Fourier transforms to quadratic surfaces and decay
   of solutions of wave equations},
   journal={Duke Math. J.},
   volume={44},
   date={1977},
   number={3},
   pages={705--714},
   issn={0012-7094},
}
\end{thebibliography}
\end{document}